\documentclass[12pt,reqno]{amsart}

\usepackage{amssymb,latexsym}

\usepackage{enumerate}
\allowdisplaybreaks
\usepackage[french,english]{babel}
\usepackage{amsmath}
\usepackage{graphicx}
\usepackage{amssymb}
\usepackage{bbm}
\usepackage{amsthm,mathtools}
\usepackage{ulem}
\usepackage{geometry}
\usepackage{tikz-cd}
\usepackage{mathrsfs}
\usepackage[colorinlistoftodos]{todonotes}
\usepackage{enumitem}
\usepackage{verbatim}
\usepackage[foot]{amsaddr}
\usepackage{dsfont}
\usepackage{cite}
\usepackage[T1]{fontenc}

\makeatletter

\@namedef{subjclassname@2010}{
	
	\textup{2020} Mathematics Subject Classification}

\makeatother
\newtheorem{thm}{Theorem}[section]
\newtheorem*{thm*}{Theorem}

\newtheorem{lem}[thm]{Lemma}

\theoremstyle{definition}

\numberwithin{equation}{section}

\newcommand{\bg}{\big}

\newcommand{\bgg}{\bigg}

\newcommand{\Bg}{\Big}

\newcommand{\Bgg}{\Bigg}

\newcommand{\lo}{\log_2}
\newcommand{\lt}{\log_3}
\newcommand{\inv}{^{-1}}

\newcommand{\mbc}{\mathbb{C}}

\newcommand{\mbr}{\mathbb{R}}

\newcommand{\mbq}{\mathbb{Q}}

\newcommand{\mch}{\mathcal{H}}

\newcommand{\mcn}{\mathcal{N}}

\newcommand{\mcm}{\mathcal{M}}

\newcommand{\mmd}{\mathrm{d}}
\newcommand{\mme}{\mathrm{e}}
\newcommand{\mmi}{\mathrm{i}}
\newcommand{\mit}{\mathrm{i}t}
\newcommand{\mdt}{\mathrm{d}t}
\newcommand{\re}{\operatorname{Re}}
\newcommand{\im}{\operatorname{Im}}

\newcommand{\whp}{\widehat{\Phi}}

\usepackage{hyperref}
\hypersetup{colorlinks=true,linkcolor=blue,anchorcolor=blue,citecolor=blue}
\usepackage{color}
\newcommand{\newabstract}[1]{%
	\par\bigskip
	\csname otherlanguage*\endcsname{#1}%
	\csname captions#1\endcsname
	\item[\hskip\labelsep\scshape\abstractname.]
}

\begin{document}

	\baselineskip=17pt

	\title[Q. Yang, S. Zhao and G.-L. Zhou]{Large values of the Hurwitz zeta function with rational parameter}

	\author{Qiyu Yang\textsuperscript{1}}
    \author{Shengbo Zhao\textsuperscript{2}}
    \author{Guang-Liang Zhou\textsuperscript{3}}
	\address{1. School of Mathematics and Statistics, Henan Normal University, Xinxiang 453007, CHINA}
	\address{2. School of Mathematical Sciences, Key Laboratory of Intelligent Computing and Applications (Tongji University), Ministry of Education, Tongji University, Shanghai 200092, China}
    \address{3. Department of Applied Mathematics, Nanjing Forestry University, Nanjing 210037, People's Republic of China}
	\email{qyyang.must@gmail.com}
	\email{shengbozhao@hotmail.com}
    \email{guangliangzhou@126.com}

	\begin{abstract} 
	   In this paper, we establish lower bounds for large values of the Hurwitz zeta function with rational parameter when the real part \(\sigma \in [1/2,1]\). These results improve the earlier results of Ramachandra and Sankaranarayanan in 1989. On the critical line, our result recovers the corresponding lower bound of de la Bretèche and Tenenbaum (2019) for the Riemann zeta function, while for \(1/2<\sigma\le 1\), our lower bounds attain the same order as the corresponding lower bounds for large values of the Riemann zeta function. Our proofs are based on the resonance method.
	\end{abstract}
	
    \keywords{Large values, the Hurwitz zeta function, the resonance method, GCD sums, smooth numbers}
	
	\subjclass[2020]{Primary 11M35, 11M06, 11N37.}
	
	\maketitle

\section{Introduction}
\label{sec-int}

The Hurwitz zeta function is one of the classical and most natural generalizations of the Riemann zeta function. The presence of the shift parameter makes it a natural object for investigating how the analytic behavior of a zeta function depends on the arithmetic nature of the parameter. Let \(0<\alpha\leq 1\) and \(s=\sigma+\mit\). For \(\sigma>1\), the Hurwitz zeta function is defined by
\[\zeta(s,\alpha)=\sum_{n=0}^\infty \frac{1}{(n+\alpha)^s}.\]
It admits a meromorphic continuation to the whole complex plane, with a unique simple pole at \(s=1\) of residue \(1\), and the Riemann zeta function is recovered by taking \(\alpha=1\). Moreover, we have \(\zeta(s,1/2) = (2^s-1)\zeta(s)\).

The arithmetic nature of \(\alpha\) leads to an important distinction. When \(\alpha=a/q\in(0,1)\) is rational and reduced, the orthogonality of Dirichlet characters gives, for \(\sigma>1\),
\begin{align}
    \label{fundamental-decomposition}
        \zeta\Bg(s,\frac{a}{q}\Bg) =q^s \sum_{\substack{n \ge 1 \\n \equiv a \pmod q}}\frac{1}{n^s} = \frac{q^s}{\phi(q)}\sum_{\chi \pmod q} \overline{\chi(a)}L(s,\chi),
\end{align}
where \(\phi(n)\) denotes Euler’s totient function. By meromorphic continuation, \eqref{fundamental-decomposition} holds throughout the whole complex plane . Thus, the Hurwitz zeta function with rational parameter is closely connected with Dirichlet \(L\)-functions. No analogous finite decomposition is available for a general irrational parameter. This difference is reflected in several aspects of the value distribution of \(\zeta(s,\alpha)\), and is particularly relevant to the study of its large values.

Motivated by \cite{Balasubrmanian1977PIAS}, Ramachandra and Sankaranarayanan \cite{Ramachandra1989ArchMath} established lower bounds for large values of the Hurwitz zeta function both for rational parameters and for a certain class of irrational parameters. More precisely, let \(a\) and \(q\) be positive integers of the same order, not exceeding \((\lo H)^B\) where \(B>1\) is fixed, and suppose that \(1000\lo T \le H \le T\), Ramachandra and Sankaranarayanan \cite[Theorem 3]{Ramachandra1989ArchMath} proved that
\begin{align}
    \label{RS89-critical-line}
    \max_{T \le t \le T+H}\Bg|\zeta\Bg(\frac{1}{2}+\mit,\frac{a}{q} \Bg)\Bg| > \exp \Bg(\frac{1}{17\mme^4} \Bg(\frac{\log H}{\phi(q)(\log\phi(q)+\lo H)}\Bg)^{1/2}\Bg).
\end{align}
Here and throughout this paper, we write \(\log_j\) for the \(j\)-th iterated logarithm, such as \(\lo x :=\log\log x\) and \(\lt x := \log\log\log x\).
For \(1/2<\sigma<1\), under the same assumptions \cite[Theorem 4]{Ramachandra1989ArchMath} gives 
\begin{align}
    \label{RS89-critical-strip}
    \max_{T \le t \le T+H}\Bg|\zeta\Bg(\sigma+\mit,\frac{a}{q} \Bg)\Bg| >\exp \Bg(\frac{(\log H)^{1-\sigma}}{9(\phi(q))^\sigma(\log\phi(q)+\lo H)}\Bg).
\end{align}
Furthermore, on the \(1\)-line, if \(1000\lt T \le H \le T\) and \(\phi(q) \le \lt H\), then \cite[Theorem 5]{Ramachandra1989ArchMath} shows that
\begin{align}
    \label{RS89-1-line}
    \max_{T \le t \le T+H}\Bg|\zeta\Bg(1+\mit,\frac{a}{q} \Bg)\Bg| > \frac{1}{1+2\mme^2}(\lo H)^{1/\phi(q)}-1.
\end{align}
They also investigated irrational parameters in the same paper, but their argument does not apply to arbitrary irrational shifts. Instead, they first establish a comparison principle between Hurwitz zeta functions with nearby parameters and then consider specially constructed irrational numbers admitting extraordinarily good rational approximations. The denominators of the corresponding rational approximations are chosen to grow rapidly, through iterated exponential constructions, so that the rational parameter estimates can be transferred to the limiting irrational parameter. In this way, \cite[Theorems 8-10]{Ramachandra1989ArchMath} yield large values on the critical line, in the critical strip, and on the \(1\)-line for this special class of irrational parameters. The authors explicitly point out that the case of a general irrational parameter appears considerably more difficult.

In a subsequent paper, Ramachandra and Sankaranarayanan \cite{Ramachandra1991AA} obtained more general large values by a modification of Montgomery's method. Specifically, let \(\sigma \in [1/2,1)\), \(\theta \in[0,2\pi)\) and \(\varepsilon>0\). For sufficiently large \(T\), they show that there exists
\[\frac{1}{2}T^\varepsilon \le t \le \frac{3}{2}T\]
such that
\[\re\bg(\mme^{-\mmi\theta}\zeta(\sigma+\mit,\alpha)\bg) \ge \frac{c}{1-\sigma}(\log t)^{1-\sigma},\]
where \(c>0\) is the constant specified in \cite[Theorem 1]{Ramachandra1991AA}. Furthermore, \cite[Theorem 2]{Ramachandra1991AA} yields that for \(\varepsilon_1>0\), a \(t\) in the same range such that
\[\re\bg(\mme^{-\mmi\theta}\zeta(1+\mit,\alpha)\bg) \ge\Bg(\frac{1}{2}\cos^2\frac{\theta}{2}-\varepsilon_1\Bg)\lo t.\]
These results apply without the special rapidly convergent rational approximation imposed in \cite{Ramachandra1989ArchMath}, although lower bounds are correspondingly of a different scale from those obtained there for rational parameters. 

More recently, the value distribution of the Hurwitz zeta function has again attracted attention, especially on the critical line. Sahay \cite{Sahay2023MPCambridgePS} studied the moments
\[M_k(T,\alpha):=\int_T^{2T} \Bg|\zeta\Bg(\frac{1}{2}+\mit,\alpha\Bg)\Bg|^{2k} \mmd t\]
for \(\alpha \in \mbq\). In analogy with the Riemann zeta function, he conjectured an asymptotic formula of the form
\[M_k(T,\alpha)\sim c_k(\alpha)T(\log T)^{k^2},\]
and established the corresponding results for \(k=1, 2\). The rationality of the shift plays a crucial role in the argument, since the Hurwitz zeta function can be expressed as a finite linear combination of Dirichlet \(L\)-functions via \eqref{fundamental-decomposition}.

The irrational case exhibits a markedly different behavior. Heap and Sahay \cite{Heap2025Crelle} proved a sharp upper bound for the fourth moment when \(\alpha\) is irrational with irrationality exponent strictly less than \(3\). As a consequence, they determined, for \(0\leq k\leq2\),
\[M_k(T,\alpha) \asymp T(\log T)^k.\]
Together with the results for rational parameters, this shows that the moment behavior of the Hurwitz zeta function can vary substantially with the arithmetic nature of the shift parameter.

These developments provide additional motivation for studying large values of the Hurwitz zeta function itself. In particular, for rational parameters, the connection with Dirichlet \(L\)-functions makes it natural to ask whether the modern methods developed for large values of the Riemann zeta function and Dirichlet \(L\)-functions can yield substantially stronger bounds than the classical large values above. This is the problem considered in the present paper.

Furthermore, the Hurwitz zeta function is closely related to Lerch zeta functions, defined by
\[
L(\lambda,s,\alpha) = \sum_{n = 0}^\infty \frac{\mme^{2n\pi\mmi \lambda}}{(n+\alpha)^s}, \quad \re(s)>1, \quad \lambda\in\mbr.
\]
In particular, the Hurwitz zeta function corresponds to the case \(\lambda=0\). Thus, the Lerch zeta function may be viewed as an additive-twist extension of the Hurwitz zeta function. From this perspective, the study of the value distribution and large values of the Hurwitz zeta function also provides a natural starting point for understanding analogous questions for the more general Lerch zeta function.

From now on, we assume that the parameter is rational and write \(\alpha=a/q\), where \(0<a<q\) and \((a,q)=1\). We first consider the case \(\sigma=1/2\). Our first result shows that, for every fixed rational parameter, the Hurwitz zeta function attains large values of the same exponential scale as the strongest known lower bounds for the Riemann zeta function on the critical line.

\begin{thm}
    \label{thm1}
    Let \(a/q \in (0,1)\) be fixed and reduced, and let \(\beta \in [0,1)\). Let \(c\) be a positive number less than \(\sqrt{2(1-\beta)}\). Then for sufficiently large \(T\), we have
    \[
    \max_{T^\beta \le t \le T}\Bg|\zeta\Bg(\frac{1}{2}+\mit,\frac{a}{q}\Bg)\Bg| \ge \exp 
    \bgg(c\sqrt{\frac{\log T \lt T}{\lo T}}\bgg).
    \]
\end{thm}

Compared with \eqref{RS89-critical-line}, our Theorem \ref{thm1} gives a stronger lower bound for a fixed rational parameter. Indeed, when \(H\asymp T\), our result gains an additional factor \(\sqrt{\lt T}\) in the exponent. We note, however, that \eqref{RS89-critical-line} is formulated for the shorter interval \([T,T+H]\), whereas our result concerns the interval \([T^\beta,T]\). We also compare Theorem \ref{thm1} with the result of de la Bretèche and Tenenbaum \cite{delaBreteche2019PLMS} for the Riemann zeta function. Although Theorem \ref{thm1} is stated for \(0<a/q<1\), its proof remains valid when \(a=q=1\). Since \(\zeta(s,1)=\zeta(s)\), in this case Theorem \ref{thm1} recovers \cite[Theorem 1.4]{delaBreteche2019PLMS}. More generally, Theorem \ref{thm1} shows that the same scale and leading constant persist for the Hurwitz zeta function with every fixed rational parameter when \(\sigma =1/2\). In particular, the congruence restriction arising from the rational parameter causes no loss in the leading constant.

We next turn to the critical strip \(\sigma \in (1/2,1)\). For every fixed rational parameter, we obtain a lower bound for large values of Hurwitz zeta function in this region.

\begin{thm}
    \label{thm2}
    Let \(\sigma \in(1/2,1)\) be fixed, and let \(a/q \in (0,1)\) be fixed and reduced. Let \(\kappa\) be any fixed positive real number satisfying
    \[0 < \kappa<\frac{\sigma-1/2}{\sigma(1+\lambda(\sigma))},\]
    where \(\lambda(\sigma)= \int_ 0^1 \mmd t/(2t^{-\sigma}-1)\). Then for sufficiently large \(T\), we have
    \[
    \max_{\sqrt{T} \le t \le T}\Bg|\zeta\Bg(\sigma+\mit,\frac{a}{q}\Bg)\Bg| \ge \exp\Bg( \Bg(\frac{\sigma}{1-\sigma}\kappa^{1-\sigma}+o_{\sigma,q}(1)\Bg) \frac{(\log T)^{1-\sigma}}{(\lo T)^\sigma}\Bg).
    \]
\end{thm}

Compared with \eqref{RS89-critical-strip}, Theorem \ref{thm2} gives a stronger lower bound for a fixed rational parameter. As \(H\asymp T\), the result of \cite{Ramachandra1989ArchMath} has an exponent of order \((\log T)^{1-\sigma}(\lo T)\inv\), whereas Theorem \ref{thm2} replaces the denominator \(\lo T\) by \((\lo T)^\sigma\). Thus, our result gains a factor \((\lo T)^{1-\sigma}\) in the exponent. Although Theorem \ref{thm2} is stated for \(0<a/q<1\), the same argument also applies to the case \(a=q=1\). In this case, the order in Theorem \ref{thm2} agrees with that of large values of the Riemann zeta function; see \cite{Aistleitner2016MathAnn,Montgomery1977CommentMH,Qyyang2024JNT}.

Finally, we turn to the \(1\)-line. For every fixed rational parameter, we obtain a lower bound for large values of Hurwitz zeta function on this line.

\begin{thm}
    \label{thm3}
    Let \(a/q \in (0,1)\) be fixed and reduced. Then there exists a positive constant \(C_q>0\), depending only on \(q\), such that, for sufficiently large \(T\), we have
    \[
    \max_{\sqrt{T}\le t\le T}\Bg|\zeta\Bg(1+\mit,\frac{a}{q}\Bg)\Bg| \ge \mme^\gamma (\lo T+\lt T) -C_{q}.
    \]
    Here \(\gamma\) denotes the Euler–Mascheroni constant.
\end{thm}
Compared with \eqref{RS89-1-line}, our Theorem \ref{thm3} gives a substantial improvement for a fixed rational parameter. Indeed, when \(H\asymp T\) and \(q\) is fixed, \eqref{RS89-1-line} gives a lower bound of order \((\lo T)^{1/\phi(q)}\), whereas Theorem \ref{thm3} reaches the order \(\lo T\) and captures the secondary term \(\lt T\).
Although Theorem \ref{thm3} is stated for \(0<a/q<1\), its proof also remains valid when \(a=q=1\).  In this case, the order in Theorem \ref{thm3} agrees with that of large values of the Riemann zeta function; see \cite[Theorem 1]{Aistleitner2019IMRN}. In fact, for \(0<\beta<1\), the interval over which the maximum is taken can be replaced by \([T^\beta,T]\). In this case, the constant \(C_q=C_{q,\beta}\) depends on both \(q\) and \(\beta\).

In this paper, the resonance method plays an important role. It can be traced back to the work of Voronin \cite{Voronin1988IANSSSR}, and it was later developed into a particularly simple and effective form by Soundararajan \cite{Soundararajan2008MathAnn}. Aistleitner \cite{Aistleitner2016MathAnn} further developed this method by combining ideas from the GCD sum, obtaining improved lower bounds for the Riemann zeta function inside the critical strip. Subsequently, Bondarenko and Seip \cite{Bondarenko2017DukeMJ,Bondarenko2018MathAnn} introduced and developed the long resonance method, establishing a deeper connection between large values of the Riemann zeta function and the GCD sum and obtaining substantial improvements, particularly on the critical line. Since then, the resonance method and its variants have been successfully applied to a wide range of problems concerning large values of the Riemann zeta function and Dirichlet \(L\)-functions. For further details and subsequent developments, we recommend \cite{Aistleitner2019QJMath,Bondarenko2018operTAA,Bondarenko2023BLMS,Ddyang2024BLMS} and the references therein.

Throughout the paper, \(\delta\) and \(\varepsilon\) denote small positive constants, not necessarily the same at each occurrence. Furthermore,  we denote the Fourier transform of a function \(f \in L^1(\mbr)\) as
\[\widehat{f}(\xi) := \int_{\mathbb R} f(x) \mme^{-\mmi\xi x} \mmd x.\]

The rest of the paper is organized as follows. In Section \ref{sec-strategy}, we briefly outline the proof strategy of the paper. In Sections \ref{sec-pthm1}-\ref{sec-pthm3}, we prove Theorems \ref{thm1}-\ref{thm3}, respectively.

\section{Strategy of the paper}
\label{sec-strategy}

We now outline the proof strategy of the paper. For \(\sigma=1/2\), our proof is based on the long resonance method applied to the normalized Hurwitz zeta function \(H_{a,q}(s)\), and follows the strategy used by Yang \cite{Ddyang2022Mathematika}, where a convolution formula is combined with the GCD sum in the study of large values of derivatives of the Riemann zeta function. We use the same GCD sum construction from \cite{delaBreteche2019PLMS} as the starting point. A new difficulty, however, arises in the Hurwitz zeta setting, since its Dirichlet coefficients are restricted to the fixed residue class \(a \pmod q\). Thus, the GCD set from \cite[p. 22]{delaBreteche2019PLMS} cannot be inserted directly into this paper. The main new ingredients on the critical line are Lemmas \ref{t1-lem3} and \ref{t1-lem4}, which overcome this congruence obstruction. More precisely, Lemma \ref{t1-lem3} transforms the set \(\mcm\) of \cite{delaBreteche2019PLMS} into a set \(\mcn\) whose elements all satisfy \(x \equiv 1 \pmod q\), while preserving the size of the GCD sum up to a constant factor depending only on the fixed modulus \(q\). Lemma \ref{t1-lem4} then shows that, for every \(x,y\in\mcn\), one can find positive integers \(k,\ell\) such that
\[xk=y\ell,\qquad k\equiv\ell\equiv a\pmod q,\]
while retaining the required lower bound for the corresponding weight. This allows the GCD sum result of \cite{delaBreteche2019PLMS} to be incorporated into the resonance framework for the Hurwitz zeta function in the case \(\sigma=1/2\).

For \(\sigma \in(1/2,1)\), we use the long resonance method developed by Yang \cite{Ddyang2023arxiv}. The main additional ingredient required in the Hurwitz zeta setting is Lemma \ref{t2-lem1}, where we establish a uniform truncation formula for the Hurwitz zeta function in a form more general than what is needed for Theorem \ref{thm2}. This approximation reduces the problem to a Dirichlet polynomial supported on the fixed residue class \(a\pmod q\), to which the resonator from \cite{Ddyang2023arxiv} can be applied. Using the non-negativity and the multiplicative structure of the resonator coefficients, we reduce the resulting expression to an arithmetic sum over the residue class \(a \pmod q\). We then apply the orthogonality of Dirichlet characters to separate the principal and non-principal character contributions. The principal character gives the main term, whereas the contributions from the non-principal characters are negligible. This yields the desired lower bound in Theorem \ref{thm2}.

For \(\sigma=1\), we adapt the long resonance method of Aistleitner, Mahatab, and Munsch \cite{Aistleitner2019IMRN} to the Hurwitz zeta function. Using \eqref{fundamental-decomposition}, together with a truncated Euler product approximation, we first reduce the problem to a sum over smooth integers lying in the fixed residue class \(a\pmod q\). We then apply the resonator used in \cite{Aistleitner2019IMRN}. The positivity of the relevant coefficients allows us to restrict the smoothness range to that of the resonator, and the resulting sum is treated by the orthogonality of Dirichlet characters. The principal character gives the main contribution, which is evaluated by Mertens' theorem, whereas the contributions from the non-principal characters remain bounded. This yields the lower bound in Theorem \ref{thm3}.

\section{Proof of Theorem \ref{thm1}}
\label{sec-pthm1}

In this section, we will use the resonator method developed in \cite{delaBreteche2019PLMS} to prove Theorem \ref{thm1}. For this purpose, we need to construct a double-version convolution formula in order to establish a connection with the GCD sum. For convenience, we define
\[H(s) := H_{a,q}(s) = q^{1/2-s}\zeta\Bg(s,\frac{a}{q}\Bg).\]
It is easy to see that the function \(H(s)\) has a unique simple pole at \(s=1\), with residue \(q^{-1/2}\). Furthermore, we have
\begin{align}
    \label{t1-moduloeq}
    \Bg|H\Bg(\frac{1}{2}+\mit\Bg)\Bg| = \Bg|\zeta\Bg(\frac{1}{2}+\mit,\frac{a}{q}\Bg)\Bg|,
\end{align}
and by \eqref{fundamental-decomposition},
\begin{align}
    \label{t1-decomposition}
    H(s) = \frac{\sqrt{q}}{\phi(q)} \sum_{\chi \pmod q} \overline{\chi(a)}L(s,\chi).
\end{align}

\subsection{Auxiliary lemmas}
\label{t1-sec-lemma}
In this section, we present several lemmas that will play an important role in the subsequent proofs. We begin with some basic properties of the function \(H(s)\).

\begin{lem}
    \label{t1-lem1}
    Let \(a\) and \(q\) be fixed. Then the following estimates hold.
\begin{enumerate}
\item[\((\mathrm{\romannumeral 1})\)] For \(X\ge 2\), we have
\begin{align}
    \label{t1-lem1-moment}
    \int_{-X}^{X}\Bg|H\Bg(\frac{1}{2}+\mit\Bg)\Bg|^2\mdt \ll_q X \log X.
\end{align}
\item[\((\mathrm{\romannumeral 2})\)] Let \(\eta \in (0,1/2]\) be fixed. Then uniformly for all \(|t|\ge 1\) and \(\sigma \in [-\eta,1+\eta]\), we have
\begin{align}
    \label{t1-lem1-bound1}
    H(\sigma+\mit) \ll_{q,\eta}(1+|t|)^{(1-\sigma+\eta)/2}.
\end{align}
\item[\((\mathrm{\romannumeral 3})\)] Uniformly for \(|t|\ge 1\), we have
\begin{align}
    \label{t1-lem1-bound2}
    H(1+\mit)\ll_q \log(2+|t|).
\end{align}
\item[\((\mathrm{\romannumeral 4})\)] There exists a function \(G_{a,q}(z)\), holomorphic in a neighborhood of \(0\), such that
\begin{align}
    \label{t1-lem1-Laurent}
    H(1+z) = \frac{1}{\sqrt{q}z}+ G_{a,q}(z).
\end{align}
\end{enumerate}
\end{lem}

\begin{proof}
By \eqref{t1-moduloeq} and the mean-square estimate for the Hurwitz
zeta function due to Rane \cite[Theorem 2]{Rane1980JLMS},
\[\int_1^T \Bg| \zeta\Bg(\frac{1}{2}+\mit,\frac{a}{q}\Bg)\Bg|^2 \mdt \ll_q T\log T.\]
A straightforward calculation yields \eqref{t1-lem1-moment}. For \((\mathrm{\romannumeral 2})\), let \(\chi\) be a character modulo \(q\), and let \(\chi^\ast\) be the primitive character of conductor \(d \mid q\) inducing \(\chi\). If \(d >1\), then
\[L(s,\chi) = L(s,\chi^\ast)\prod_{p \mid q, p \nmid d} \Bg(1-\frac{\chi^\ast(p)}{p^s}\Bg).\]
Since \(q\) is fixed, the finite Euler product is \(O_{q,\eta}(1)\) uniformly for \(\sigma \in [-\eta,1+\eta]\). By the classical convexity bound for the primitive character (see, for example \cite[Theorem 3]{Rademacher1959MathZ}),
\[L(\sigma+\mit,\chi^\ast)\ll_{q,\eta}(1+|t|)^{(1-\sigma+\eta)/2}.\]
Hence the same estimate holds for \(L(\sigma+\mit,\chi)\). If \(d=1\), then \(\chi=\chi_0\) is the principal character modulo \(q\), and
\[L(s,\chi_0) = \zeta(s)\prod_{p \mid q}\bg(1-p^{-s}\bg).\]
The corresponding convexity bound for \(\zeta(s)\) therefore gives
\[L(\sigma+\mit,\chi_0)\ll_{q,\eta}(1+|t|)^{(1-\sigma+\eta)/2}.\]
Thus, uniformly for all characters modulo \(q\),
\[L(\sigma+\mit,\chi)\ll_{q,\eta}(1+|t|)^{(1-\sigma+\eta)/2}.\]
Combining this with \eqref{t1-decomposition} proves \eqref{t1-lem1-bound1}. Similarly, the standard estimate, valid for every Dirichlet character \(\chi\) modulo \(q\),
\[L(1+\mit,\chi)\ll_q\log(2+|t|)\]together with \eqref{t1-decomposition} gives \eqref{t1-lem1-bound2}. Finally, \eqref{t1-lem1-Laurent} follows from the Laurent expansion of \(H(s)\) at \(s=1\). We complete the proof of Lemma \ref{t1-lem1}.
\end{proof}

We next establish a double-version convolution formula. The argument is similar to \cite[Lemma 3]{Ddyang2022Mathematika} and \cite[Lemma 2.2]{Zhli2026JAustMS}.

\begin{lem}
    \label{t1-lem2}
Let \(\sigma \in [0,1)\) be fixed and suppose \(z = x+\mmi y\). Assume that \(F\) is holomorphic in the horizontal strip \(\sigma-2 \le y \le 0\), satisfying
\[\max_{\sigma-2\le y \le 0}|F(x+\mmi y)|=O\Big(\frac{1}{x^{2} + 1}\Big).\]
Then, for every real \(t \neq 0\),
\[\int_{-\infty}^\infty H(\sigma+\mit+\mmi y)H(\sigma-\mit+\mmi y)F(y)\mmd y= q\sum_{\substack{k,\ell \ge 1 \\ k \equiv \ell \equiv a \pmod q}}\frac{\widehat{F}(\log k\ell)}{k^{\sigma+\mit}\ell^{\sigma-\mit}}-P(t),\]
where 
\[P(t)=\frac{2\pi}{\sqrt{q}}(H(1-2\mit)F(-t-\mmi(1-\sigma))+H(1+2\mit)F(t-\mmi(1-\sigma))).\]
\end{lem}

\begin{proof}
Write \(h(z)=H(z+\mit)H(z-\mit)F(\mmi\sigma-\mmi z)\). Then the only poles in \(\sigma \le \re(z) \le 2\) occur at \(1 \pm \mit\). Let \(Y\) be large. Integrating \(h(z)\) along the rectangle with vertices \(\sigma \pm \mmi Y\) and \(2 \pm \mmi Y\), and using the residue theorem, gives
    \[J_1+J_2+J_3+J_4=\mmi P(t).\]
Here, \(J_1\) and \(J_3\) represent the contributions from the lower and upper horizontal segments of the rectangular contour, running from \(\sigma-\mmi Y\) to \(2-\mmi Y\) and from \(2+\mmi Y\) to \(\sigma+\mmi Y\), respectively. Similarly, \(J_2\) and \(J_4\) are the contributions from the right and left vertical segments, oriented from \(2-\mmi Y\) to \(2+\mmi Y\) and from \(\sigma+\mmi Y\) to \(\sigma-\mmi Y\), respectively.

For \(J_2\), we have
\begin{align*}
    J_2 &= \mmi \int_{-Y}^Y H(2+\mit+\mmi y)H(2-\mit+\mmi y)F(y-\mmi(2-\sigma))\mmd y \\
    & = q\mmi\sum_{\substack{k,\ell \ge 1 \\ k \equiv \ell \equiv a \pmod q}}\frac{1}{k^{2+\mit}\ell^{2-\mit}} \int_{-Y}^Y \mme^{-\mmi y \log k\ell}F(y-\mmi(2-\sigma))\mmd y.
\end{align*}
Put \(\omega=y-\mmi(2-\sigma)\). Then
\[J_2 = q\mmi\sum_{\substack{k,\ell \ge 1 \\ k \equiv \ell \equiv a \pmod q}}\frac{1}{k^{\sigma+\mit}\ell^{\sigma-\mit}} \int_{-Y-\mmi(2-\sigma)}^{Y-\mmi(2-\sigma)}\mme^{-\mmi\omega\log k \ell}F(\omega)\mmd \omega.\]
Since \(F\) is holomorphic in the strip \(\sigma-2\le \im\omega\le0\) and satisfies the growth condition, Cauchy’s theorem allows us to shift the line of integration to the real axis. Therefore,
\[\lim_{Y \to \infty}J_2=q\mmi\sum_{\substack{k,\ell \ge 1 \\ k,\ell \equiv a \pmod q}}\frac{\widehat{F}(\log k\ell)}{k^{\sigma+\mit}\ell^{\sigma-\mit}}.\]
On the other hand, 
\[\lim_{Y \to \infty}(-J_4)=\mmi\int_{-\infty}^\infty H(\sigma+\mit+\mmi y)H(\sigma-\mit+\mmi y)F(y)\mmd y.\]

 To estimate \(J_1\), using the growth condition on \(F\) together with \eqref{t1-lem1-bound1}, we have
\begin{align*}
        |J_1| &\ll \frac{1}{Y^2}\Bg(\Bg(\int_\sigma^{1+\eta}+\int_{1+\eta}^2\Bg)|H(x-\mmi Y+\mit)||H(x-\mmi Y-\mit)|\mmd x\Bg) \\
        &\ll_\eta \frac{1}{Y^2}\Bg(1+ \int_\sigma^{1+\eta}\bg(1+Y^{1-x+\eta}\bg)\mmd x\Bg) \ll\frac{1}{Y^2} \Bg(1+\frac{Y^{1-\sigma+\eta}-1}{\log Y} \Bg).
\end{align*}
Taking \(\eta=1/6\), for sufficiently large \(Y\),
    \[J_1 \ll \frac{1}{T^{5/6}\log Y} \to 0.\]
Similarly, we obtain \(J_3 \to 0\) as \(Y \to \infty\). This completes the proof of Lemma \ref{t1-lem2}.    
\end{proof}

For a finite set \(\mcm\) of positive integers and \(\theta>0\), we define
\[S_\theta(\mcm) := \sum_{m,n\in \mcm}\Bg(\frac{(m,n)}{[m,n]}\Bg)^\theta.\]
Here, \((m,n)\) and \([m,n]\) denote the greatest common divisor and least common multiple of \(m\) and \(n\), respectively. We write \(S(\mcm)=S_{1/2}(\mcm)\) and, for convenience, put 
\[Q(m,n):=\frac{(m,n)}{\sqrt{mn}}=\sqrt{\frac{(m,n)}{[m,n]}}.\]
By the construction of de la Bretèche and Tenenbaum \cite{delaBreteche2019PLMS}, for every sufficiently large \(N\), there exists a set \(\mcm\) of positive integers with \(|\mcm|\le N\) such that
\begin{align}
    \label{t1-gcdsum-M}
    \frac{S(\mcm)}{|\mcm|}\ge \exp\bgg((2\sqrt{2}-\delta)\sqrt{\frac{\log N \lt N}{\lo N}}\bgg),
\end{align}
where \(\delta>0\) can be arbitrarily small. Moreover, the prime factors of the elements of \(\mcm\) satisfy 
\[\max_{m\in\mcm}P^+(m) \le (\log N)^{1+\nu}\]
for every fixed \(\nu>0\), provided that \(N\) is sufficiently large. Here, \(P^+(y)\) denotes the largest prime factor of \(y\). Since \(q\) is fixed, we may also assume, by taking \(N\) sufficiently large, that every prime factor of every \(m\in\mcm\) is larger than \(q\).

For each \(m\in\mcm\), let \(u_m \in [1,q]\) be the unique representative satisfying
\[u_m m \equiv 1 \pmod q.\]
We then define
\[\mcn :=\{u_m m:m\in\mcm\}.\]
Thus, every element of \(\mcn\) belongs to the residue class \(1 \pmod q\). The following lemma collects the properties of \(\mcn\) that will be needed later.

\begin{lem}
    \label{t1-lem3}
    The set \(\mcn\) defined above satisfies the following properties.
    \begin{enumerate}
\item[\((\mathrm{\romannumeral 1})\)] The map \(m \mapsto u_m m\) from \(\mcm\) to \(\mcn\) is injective. In particular, \(|\mcn| =|\mcm|\).
\item[\((\mathrm{\romannumeral 2})\)] For every \(x \in \mcn\), we have 
\begin{align}
    \label{t1-lem3-congruence}
    x \equiv 1 \pmod q.
\end{align}
\item[\((\mathrm{\romannumeral 3})\)] We have
\begin{align}
    \label{t1-lem3-gcdsum-X}
    \frac{S(\mcn)}{|\mcn|} \ge \frac{1}{q}\frac{S(\mcm)}{|\mcm|}.
\end{align}
Moreover, the largest prime factor occurring in \(\mcn\) is at most the maximum of \(q\) and the largest prime factor occurring in \(\mcm\).
\end{enumerate}
\end{lem}

\begin{proof}
Since every prime factor of \(m\in\mcm\) is larger than \(q\), while every prime factor of \(u_m\) is at most \(q\), the factorization of \(u_m m\) uniquely determines both \(u_m\) and \(m\). Hence the map \(m \mapsto u_m m\) is injective, and therefore \(|\mcn|=|\mcm|\). \eqref{t1-lem3-congruence} follows immediately from the definition of \(u_m\). Let \(x=u_m m\) and \(y=u_n n\). Since \((m,n)\mid(x,y)\) and \(\sqrt{u_mu_n}\leq q\), we have
\[ Q(x,y) = \frac{(x,y)}{\sqrt{xy}} \geq \frac1q\frac{(m,n)}{\sqrt{mn}} = \frac{1}{q}Q(m,n).\]
Summing over \(m,n\in\mcm\) proves the desired estimate for \(S(\mcn)\). Finally, the assertion concerning the prime factors follows directly from the definition of \(\mcn\). Combining the above arguments, we complete the proof of Lemma \ref{t1-lem3}.    
\end{proof}

The preceding construction places all elements of \(\mcn\) in the class \(1 \pmod q\). In the double-version convolution formula, however, the two Dirichlet coefficients are restricted to the class \(a \pmod q\). We therefore need to show that, for every pair \(x,y\in\mcn\), the relation arising in the resonance argument can be realized by two integers \(k\) and \(\ell\) satisfying \(k \equiv \ell \equiv a\pmod q.\) This is given by the following lemma.

\begin{lem}
    \label{t1-lem4}
Let \(x,y\in \mcn\). Then there exist positive integers \(k\) and \(\ell\) such that \(xk=y\ell\), \(k \equiv \ell \equiv a \pmod q\) and 
\begin{align}
    \label{t1-lem4-bound1}
    \frac{q}{\sqrt{k\ell}}\ge Q(x,y).
\end{align}
Moreover, \(k\) and \(\ell\) may be chosen so that
\begin{align}
    \label{t1-lem4-bound2}
    k\ell \le q^2 \frac{[x,y]}{(x,y)}.
\end{align}
\end{lem}

\begin{proof}
Let \(d=(x,y)\). Then put \(x = d x^\prime\) and \(y = d y^\prime\). Since \(x \equiv y\equiv1\pmod q\), we have \((d,q)=1\), and \(x^\prime, y^\prime \equiv d\inv \pmod q\). Let \(r_d \in [1,q]\) satisfy \(r_d \equiv ad \pmod q\), and define \(k = y^\prime r_d\) and \(\ell = x^\prime r_d\). Then, 
\[xk=dx^\prime y^\prime r_d = d y^\prime x^\prime r_d = y\ell.\]
Moreover, we have
\[
k \equiv y^\prime r_d \equiv d\inv (ad) \equiv a \pmod q
\]
and similarly \(\ell \equiv a \pmod q\). Finally, 
\[\frac{q}{\sqrt{k\ell}} = \frac{q}{r_d\sqrt{x^\prime y^\prime}} = \frac{q}{r_d} \frac{d}{\sqrt{xy}} = \frac{q}{r_d}Q(x,y) \ge Q(x,y),\]
since \(r_d \le q\). This gives \eqref{t1-lem4-bound1}. Also, \eqref{t1-lem4-bound2} follows from
\[k\ell =r_d^2x^\prime y^\prime = r_d^2 \frac{[x,y]}{(x,y)} \le q^2 \frac{[x,y]}{(x,y)}.\]
We now complete the proof of Lemma \ref{t1-lem4}.  
\end{proof}

Finally, we define 
\[
K(u) := \frac{\sin^2(\varepsilon u \log T)}{\pi u^2\varepsilon \log T}
\]
with \(\varepsilon \in (0,1)\). Thus, \(K(u)\) is non-negative on \(\mbr\) and its Fourier transform satisfies
\[
\widehat{K}(\xi) = \max \Bg(0, \Bg(1-\frac{|\xi|}{2\varepsilon\log T}\Bg)\Bg).
\]
The following properties of \(K\) will be used throughout the proof.

\begin{lem}
    \label{t1-lem5}
     Let \(K\) be defined as above. If \(k\) and \(\ell\) are positive integers satisfying \(k\ell\le T^\varepsilon\), then
     \begin{align}
         \label{t1-lem5-bound1}
         \widehat{K}(\log k\ell) \ge \frac{1}{2}.
     \end{align}
     Moreover, uniformly for \(t \in \mbr\),
     \begin{align}
         \label{t1-lem5-bound2}
         \Bg|K\Bg(t\pm \frac{\mmi}{2}\Bg)\Bg| \ll_\varepsilon \frac{T^\varepsilon}{(1+t^2)\log T}.
     \end{align}
     The corresponding estimate for \(K'(t\pm \mmi/2)\) holds with at most one additional factor \(\log T\).
\end{lem}

\begin{proof}
    From the explicit formula for \(\widehat{K}\), if \(k\ell \le T^\varepsilon\), then \(\log k\ell \le \varepsilon\log T\). Hence
    \[\widehat{K}(\log k\ell) = 1 - \frac{\log k\ell}{2\varepsilon \log T} \ge \frac{1}{2},\]
    which proves the first assertion. On the other hand, let \(z = t \pm \mmi/2\). Since \(|z|^2 = t^2 + 1/4 \asymp 1+ t^2\) and \(|\sin(x+\mmi y)| \le \mme^{|y|}\), we have \(|\sin^2(\varepsilon z \log T)| \le T^\varepsilon\). Substituting this into the definition of \(K\), we obtain
    \[\Bg|K\Bg(t\pm \frac{\mmi}{2}\Bg)\Bg| \ll_\varepsilon \frac{T^\varepsilon}{(1+t^2)\log T}.\]
    Finally, differentiating the definition of \(K\) and applying the same estimates to the sine and cosine factors gives the corresponding bound for \(K'(t\pm \mmi/2)\), with at most one additional factor \(\log T\). 
\end{proof}

\subsection{Proof of Theorem \ref{thm1}}
\label{t1-sec-prove}

We now turn to the proof of Theorem \ref{thm1}. Fix \(c<\sqrt{2(1-\beta)}\). Choose \(\kappa\in(0,1-\beta)\) sufficiently close to \(1-\beta\) so that
\[c <\sqrt{2\kappa,}\]
and put \(N=\lfloor T^\kappa\rfloor\). Let \(\mcn\) be the set constructed in Lemma \ref{t1-lem3}. For an arbitrarily small fixed \(\delta >0\), \eqref{t1-lem3-gcdsum-X} gives that 
\begin{align}
    \label{t1-gcdX}
    \frac{S(\mcn)}{|\mcn|} \ge \frac{1}{q}\exp\bgg((2\sqrt{2}-\delta)\sqrt{\frac{\log N \lt N}{\lo N}}\bgg)
\end{align}
holds if \(N\) is large. For each integer \(j \ge 0\), define
\[\mcn_j := \bgg[\Bg(1 + \frac{\log T}{T}\Bg)^{j}, \Bg(1 + \frac{\log T}{T}\Bg)^{j + 1}\bgg) \cap \mcn.\]
Following the notation of \cite[p. 22]{delaBreteche2019PLMS}, let \(h_j := \min \mcn_j\) if \(\mcn_j \neq \emptyset\). Define \(\mch\) as the set of all \(h_j\) and we consider the function \(r\) on this set defined by 
\[r(h_{j})=\sqrt{\sum_{m\in\mcn_j}1}.\]
Furthermore, we set the resonator
\[R(t) := \sum_{h\in\mch}r(h)h^{-\mit}.\]
The Cauchy-Schwarz inequality implies that
\begin{align}
    \label{t1-Rupper}
    |R(t)|^2 \le R(0)^{2}\le N\sum_{h\in\mch}r(h)^{2}\le N|\mcn|.
\end{align}
As in \cite{Bondarenko2017DukeMJ}, we take \(\Phi(t) := \mme^{-t^2/2}\), whose Fourier transform satisfies \(\widehat{\Phi}(\xi)=\sqrt{2\pi}\Phi(\xi)>0\) for all \(\xi\in\mbr\). According to \cite[Lemma 5]{Bondarenko2018MathAnn}, we obtain
\begin{align}
    \label{t1-intRPbound}
    \int_\mbr |R(t)|^{2}\Phi\Bg(\frac{t\log T}{T}\Bg)\mmd t\ll\frac{|\mcn|T}{\log T}.
\end{align}

Fix \(\varepsilon>0\) such that \(\kappa+3\varepsilon<1\), define the function \(K\) as in Section \ref{t1-sec-lemma}. Next, define
\[\Upsilon(t,y):= H\Bg(\frac{1}{2}+\mit +\mmi y \Bg)H\Bg(\frac{1}{2}-\mit+\mmi y\Bg) K(y)\]
and
\[I(T) := \int_{|t|>0} |R(t)|^{2}\Phi\Bg(\frac{t\log T}{T}\Bg)\int_\mbr \Upsilon(t,y) \mmd y \mmd t.\]
The exclusion of \(t=0\) is only needed in order to apply the double-version convolution formula, that is, Lemma \ref{t1-lem2}. As will be shown later, after the two pole terms are combined, the apparent singularity at \(t=0\) is removable by \eqref{t1-lem1-Laurent}. Taking \(\sigma=1/2\) and \(\eta=1/30\) in Lemma \ref{t1-lem1} \((\mathrm{\romannumeral 2})\), and enlarging the implied constant to cover the range where \(|t\pm y| <1\), we obtain
\begin{align}
    \label{t1-Hbound}
    \Bg|H\Bg(\frac{1}{2}\pm \mit +\mmi y \Bg)\Bg| \ll (1+|t|+|y|)^{3/10}.
\end{align}
Following the approach of \cite{Bondarenko2018MathAnn,delaBreteche2019PLMS,Ddyang2022Mathematika}, we now show that the main contribution to \(I(T)\) comes from the range \(2T^\beta \le |t| \le T/2\) and \(|y| \le |t|/2\). To this end, we first consider \(0<|t|\le2T^\beta\). Uniformly for \(|y| \le T^\beta\), \eqref{t1-lem1-moment} yields that
\begin{align*}
        \int_{|t|\le 2T^\beta}& \Bg|H\Bg(\frac{1}{2}+\mit+\mmi y\Bg)H\Bg(\frac{1}{2}-\mit+\mmi y\Bg)\Bg|\mmd t \\
    &\le\Bg(\int_{|u|\le3T^\beta}\Bg|H\Bg(\frac{1}{2}+\mmi u\Bg)\Bg|^2\mmd u\Bg)^{1/2} \Bg(\int_{|v|\le3T^\beta}\Bg|H\Bg(\frac{1}{2}+\mmi v\Bg)\Bg|^2\mmd v\Bg)^{1/2} \\
    &\ll_q T^\beta \log T.
\end{align*}
Thus, by the non-negativity of \(K\), we obtain
\begin{align}
    \label{t1-tail1-1}
    \Bg|\int_{0<|t|\le 2T^\beta} \int_{|y|\le T^\beta} \Upsilon(t,y) \mmd y \mmd t\Bg| \ll_q T^\beta \log T.
\end{align}
For \(|y| > T^\beta\), by \eqref{t1-Hbound} and the definition of \(K\), we obtain
\begin{align}
    \label{t1-tail1-2}
    \Bg|\int_{0<|t|\le 2T^\beta} &\int_{|y|> T^\beta} \Upsilon(t,y) \mmd y \mmd t \Bg|\nonumber \\ 
    &\ll \int_{0<|t|\le 2T^\beta} \int_{|y|> T^\beta} (1+|t|+|y|)^{3/5}\frac{ \mmd y \mmd t}{(|t|+|y|)^2} \ll T^{3\beta/5}.
\end{align}
Thus, combining \eqref{t1-tail1-1} and \eqref{t1-tail1-2} with \eqref{t1-Rupper} gives
\begin{align}
    \label{t1-tail1}
    \Bg|\int_{0<|t|\le 2T^\beta} |R(t)|^{2}\Phi\Bg(\frac{t\log T}{T}\Bg)& \int_\mbr \Upsilon(t,y) \mmd y \mmd t\Bg| \nonumber\\
    &\ll R(0)^2 T^\beta \log T \ll |\mcn|T^{\beta+\kappa}\log T.
\end{align}
Then, for \(|t| >T/2\), by the exponential decay of \(\Phi\), together with \eqref{t1-Hbound}, it follows that
\begin{align}
    \label{t1-tail2}
    \Bg|\int_{|t|>T/2} |R(t)|^{2}\Phi\Bg(\frac{t\log T}{T}\Bg)\int_\mbr \Upsilon(t,y) \mmd y \mmd t\Bg| = o(|\mcn|).
\end{align}
Finally, applying \eqref{t1-intRPbound} and \eqref{t1-Hbound}, we have
\begin{align}
    \label{t1-tail3}
    \Bg|\int_{2T^\beta \le |t| \le T/2}|R(t)|^{2}\Phi\Bg(\frac{t\log T}{T}\Bg) \int_{|y|>|t|/2}\Upsilon(t,y) \mmd y \mmd t\Bg| \ll|\mcn| \frac{T^{1-2\beta/5}}{(\log T)^2}.
\end{align}
Throughout the above argument, we repeatedly invoked Lemma \ref{t1-lem3} \((\mathrm{\romannumeral 1})\), which states that \(|\mcn|=|\mcm|\). Putting the preceding estimates together, we arrive at
\begin{align*}
    I(T) &= \int_{2T^\beta \le |t| \le T/2}|R(t)|^{2}\Phi\Bg(\frac{t\log T}{T}\Bg) \int_{|y|\le |t|/2}\Upsilon(t,y) \mmd y \mmd t \\
    &\,\,\,+ O\bg(|\mcn|T^{\beta+\kappa} \log T\bg)+O\Bg(|\mcn| \frac{T^{1-2\beta/5}}{(\log T)^2}\Bg).
\end{align*}
Since \(H(\overline{s})=\overline{H(s)}\), we have
\[\Bg|H\Bg(\frac{1}{2}+\mmi u\Bg)\Bg| = \Bg|H\Bg(\frac{1}{2}-\mmi u\Bg)\Bg|.\]
Moreover, in the range \(2T^\beta\le |t|\le T/2\) and \(|y|\le |t|/2\), we have \(T^\beta\le |t\pm y|\le T\). Combining this with \eqref{t1-intRPbound} gives 
\begin{align}
\label{t1-max1}
    I(T)\ll \frac{|\mcn|T}{\log T} \max_{T^\beta \le t \le T}\Bg|H\Bg(\frac{1}{2}+\mit\Bg)\Bg|^2 + O\bg(|\mcn|T^{\beta+\kappa} \log T\bg)+O\Bg(|\mcn| \frac{T^{1-2\beta/5}}{(\log T)^2}\Bg).
\end{align}

Applying Lemma \ref{t1-lem2} with \(\sigma=1/2\) and \(F=K\), we obtain
\begin{align}
    \label{t1-Iequation}
    I(T)=I_1(T)+I_2(T),
\end{align}
where
\[I_1(T)=\int_\mbr q \sum_{\substack{k,\ell \ge 1 \\ k \equiv \ell \equiv a \pmod q}}\frac{\widehat{K}(\log k\ell)}{\sqrt{k\ell}}\Bg(\frac{\ell}{k}\Bg)^{\mit}|R(t)|^{2}\Phi\Bg(\frac{t\log T}{T}\Bg)\mmd t\]
and
\[I_2(T)=-\int_\mbr P(t)|R(t)|^{2}\Phi\Bg(\frac{t\log T}{T}\Bg)\mmd t.\]
Now we show that under \(\kappa+3\varepsilon<1\), we have
\begin{align}
    \label{t1-I2upper}
    |I_2(T)| = o_{q,\varepsilon}\Bg(\frac{|\mcn|T}{\log T}\Bg).
\end{align}
For \(|t|\ge 1\), by \eqref{t1-lem1-bound2}, \eqref{t1-lem5-bound2} and \eqref{t1-Rupper}, we have
\begin{align*}
    \int_{|t|\ge 1}&|P(t)||R(t)|^{2}\Phi\Bg(\frac{t\log T}{T}\Bg)\mmd t \\
    & \ll_{q,\varepsilon} \frac{|\mcn|T^{\kappa+\varepsilon}}{\log T} \int_{|t|\ge 1}\frac{\log(2+t^2)}{1+t^2}\mmd t \ll_{q,\varepsilon} \frac{|\mcn|T^{\kappa+\varepsilon}}{\log T}.
\end{align*}
For \(0<|t|<1\), by \eqref{t1-lem1-Laurent}, 
\[H(1\mp 2\mit)=\pm \frac{\mmi}{2\sqrt{q}t}+G_{a,q}(\mp 2\mit).\]
Hence, the only potentially singular contribution to \(P(t)\) is a constant multiple of
\[\frac{1}{t}\Bg(K\Bg(t-\frac{\mmi}{2}\Bg)-K\Bg(-t-\frac{\mmi}{2}\Bg)\Bg).\]
By the mean value theorem and the derivative estimate for \(K\) in Lemma \ref{t1-lem5}, we have
\[\Bg|\frac{1}{t}\Bg(K\Bg(t-\frac{\mmi}{2}\Bg)-K\Bg(-t-\frac{\mmi}{2}\Bg)\Bg)\Bg|\le 2\sup_{|u|\le1}\Bg|K^\prime\Bg(u-\frac{\mmi}{2}\Bg)\Bg| \ll_\varepsilon T^\varepsilon.\]
Since \(G_{a,q}\) is holomorphic in a neighborhood of \(0\), the remaining terms are also \(O_{q,\varepsilon}(T^\varepsilon)\). Therefore, for \(0<|t|<1\), \(P(t)\ll_{q,\varepsilon}T^\varepsilon\). Consequently, by \eqref{t1-Rupper} and \(\kappa+3\varepsilon<1\), we get
\[\int_{0<|t|< 1}|P(t)||R(t)|^{2}\Phi\Bg(\frac{t\log T}{T}\Bg)\mmd t= o_{q,\varepsilon}\Bg(\frac{|\mcn|T}{\log T}\Bg).\]
It follows that \eqref{t1-I2upper} holds from the above estimates.

It remains to estimate \(I_1(T)\). Expanding \(|R(t)|^2\) implies that
\[I_1(T)=\sqrt{2\pi}\frac{T}{\log T}\sum_{h,g \in \mch}r(h)r(g)  \sum_{\substack{k,\ell \ge 1 \\ k \equiv \ell \equiv a \pmod q}}\frac{q\widehat{K}(\log k\ell)}{\sqrt{k\ell}}\Phi\Bg(\frac{T}{\log T}\log \frac{g\ell}{hk}\Bg). \]
Since \(\widehat K(\log k\ell)=0\) for \(k\ell>T^{2\varepsilon}\), the inner sum is finite. As all terms are non-negative, we may restrict to \(k\ell \le T^\varepsilon\). In fact, by \eqref{t1-lem5-bound1}, we have
\[I_1(T) \gg\frac{T}{\log T}\sum_{\substack{k,\ell \ge 1,k\ell \le T^\varepsilon \\ k \equiv \ell \equiv a \pmod q}}\frac{q}{\sqrt{k\ell}}\sum_{h,g \in \mch}r(h)r(g) \Phi\Bg(\frac{T}{\log T}\log \frac{g\ell}{hk}\Bg). \]
Following \cite{delaBreteche2019PLMS}, let \(x\in\mcn_i\) and \(y\in\mcn_j\), and suppose that \(xk=y\ell\). Then set \(h_i = \min\mcn_i\) and \(h_j = \min\mcn_j\). From the definition of the sets \(\mcn_i\) and \(\mcn_j\), we obtain
\[1\le \frac{x}{h_i}<1+\frac{\log T}{T},\quad1\le \frac{y}{h_j}<1+\frac{\log T}{T}.\]
Thus, \(xk=y\ell\) yields 
\[\Bg|\frac{T}{\log T}\log\frac{h_j\ell}{h_ik}\Bg|\le 2\frac{T}{\log T}\log\Bg(1+\frac{\log T}{T}\Bg)\le2.\]
Furthermore,
\[\Phi\Bg(\frac{T}{\log T}\log\frac{h_j\ell}{h_ik}\Bg)\gg 1.\]
For fixed \(k,\ell,i,j\), the number of pairs \((x,y) \in \mcn_i\times \mcn_j\) satisfying \(xk=y\ell\) is at most
\[\min\{|\mcn_i|,|\mcn_j|\} \le \sqrt{|\mcn_i||\mcn_j|}=r(h_i)r(h_j).\]
It follows that
\[I_1(T) \gg\frac{T}{\log T}\sum_{x,y\in\mcn}\sum_{\substack{k,\ell \ge 1,k\ell \le T^\varepsilon\\ xk=y\ell \\ k \equiv \ell \equiv a \pmod q}}\frac{q}{\sqrt{k\ell}}.\]

Applying Lemma \ref{t1-lem4} and using the construction of \(k\) and \(\ell\), we deduce that
\[I_1(T) \gg_q \frac{T}{\log T} \sum_{\substack{x,y\in\mcn \\ [x,y]/(x,y)\le T^\varepsilon/q^2}} Q(x,y). \]
Furthermore, Rankin's trick shows that
\begin{align}
    \label{t1-I1lower}
    I_1(T) \gg\frac{T}{\log T}\bg(S(\mcn) -C_qT^{-\varepsilon/6}S_{1/3}(\mcn)\bg).
\end{align}
Here, \(C_q\) denotes a constant that depends only on \(q\). Let \(Y\) denote the largest prime factor occurring among the elements of \(\mcn\). Choose and fix \(0<\nu<1/2\). By the construction of \(\mcn\), and since \(q\) is fixed, we have \(Y \le (\log T)^{1+\nu}\) for sufficiently large \(T\). Proceeding as in \cite[p. 25]{delaBreteche2019PLMS}, we obtain
\[\frac{S_{1/3}(\mcn)}{|\mcn|} \ll_q \exp\bg(Y^{2/3}\bg).\]
Since \(\nu<1/2\), \(Y^{2/3} \le (\log T)^{2(1+\nu)/3} =o(\log T)\), and thus, we have \(\exp\bg(Y^{2/3}\bg)=T^{o(1)}\). It follows that
\[C_qT^{-\varepsilon/6}S_{1/3}(\mcn) \ll |\mcn|T^{-\varepsilon/6}\exp\bg(Y^{2/3}\bg)=|\mcn|T^{-\varepsilon/6+o(1)}=o(|\mcn|).\]
Therefore, \eqref{t1-I1lower} yields
\begin{align}
    \label{t1-I1lower2}
    I_1(T) \gg_q \frac{T}{\log T}S(\mcn).
\end{align}

Combining \eqref{t1-max1}, \eqref{t1-Iequation}, \eqref{t1-I2upper} and \eqref{t1-I1lower2}, we have
\begin{align}
    \label{t1-max2}
    \max_{T^\beta \le t \le T}\Bg|H\Bg(\frac{1}{2}+\mit\Bg)\Bg|^2 \gg_q \frac{S(\mcn)}{|\mcn|}
\end{align}
for sufficiently large \(T\). Using \eqref{t1-gcdsum-M} and \eqref{t1-lem3-gcdsum-X}, together with \eqref{t1-max2}, we obtain
\[\max_{T^\beta \le t \le T}\Bg|H\Bg(\frac{1}{2}+\mit\Bg)\Bg| \ge \exp\bgg((\sqrt{2}-\delta)\sqrt{\frac{\log N \lt N}{\lo N}}\bgg).\]
Recalling that \(N=\lfloor T^\kappa\rfloor\), we have
\[\max_{T^\beta \le t \le T}\Bg|H\Bg(\frac{1}{2}+\mit\Bg)\Bg| \ge\exp\bgg(\bg(\bg(\sqrt{2}-\delta\bg)\sqrt{\kappa}+o(1)\bg)\sqrt{\frac{\log T \lt T}{\lo T}}\bgg).\]
Since \(c<\sqrt{2\kappa}\), we may choose \(\delta>0\) sufficiently small so that,
\[c<\bg(\sqrt{2}-\delta\bg)\sqrt{\kappa}.\]
Hence, for sufficiently large \(T\),
\[\max_{T^\beta \le t \le T}\Bg|H\Bg(\frac{1}{2}+\mit\Bg)\Bg| \ge\exp\bgg(c\sqrt{\frac{\log T \lt T}{\lo T}}\bgg).\]
Finally, by \eqref{t1-moduloeq}, we complete the proof of Theorem \ref{thm1}.

\section{Proof of Theorem \ref{thm2}}
\label{sec-pthm2}
In this section, we prove Theorem \ref{thm2} by using the long resonance method in \cite{Ddyang2023arxiv}.

\subsection{An auxiliary lemma}
\label{t2-sec-lemma}
We first establish a uniform truncation formula for the Hurwitz zeta function. This result relies on the properties of the periodic Bernoulli function and its Fourier transform.
\begin{lem}
    \label{t2-lem1}
Let \(\sigma_0>0\) and \(A>1\) be fixed, and put \(N=\lfloor T^A\rfloor\). Then there exists a constant \(\delta=\delta(\sigma_0,A)>0\) such that, uniformly for \(\sigma\ge\sigma_0\), \(1\le |t|\le T\) and \(0 <\alpha \le 1\), we have
\[\zeta(s,\alpha)=\sum_{n=0}^{N-1} \frac{1}{(n+\alpha)^s}+\frac{(N+\alpha)^{1-s}}{s-1}+O_{\sigma_0,A}(T^{-\delta}),\]
where \(s=\sigma+\mit\).
\end{lem}

\begin{proof}
For \(z \in \mbc\) and \(j \ge 0\), set \((z)_0=1\) and \((z)_j = z(z+1)\cdots(z+j-1)\). Moreover, let \(B_j(x)\) be the \(j\)-th Bernoulli polynomial and let \(\widetilde{B}_j(x) = B_j(x-\lfloor x \rfloor)\) be its periodic version. Choose an integer \(M\ge 1\) sufficiently large such that
\[A\sigma_0+2M(A-1)>1.\]
By \cite[Chapter 12, Theorem 12.21]{Apostol1976}, we have
\begin{align}
    \label{t2-lem1-Euler}
    \zeta(s,\alpha)=\sum_{n=0}^{N-1}\frac{1}{(n+\alpha)^s}+\frac{(N+\alpha)^{1-s}}{s-1}+\frac{1}{2(N+\alpha)^s}-s\int_N^\infty\frac{\widetilde{B}_1(x)}{(x+\alpha)^{s+1}}\mmd x.
\end{align}

For \(k\ge 1\), define
\[I_k = \frac{1}{k!}\int_N^\infty\frac{\widetilde{B}_k(x)}{(x+\alpha)^{s+k}}\mmd x.\]
Since
\[\frac{\mmd}{\mmd x}\Bg(\frac{\widetilde{B}_{k+1}(x)}{(k+1)!}\Bg) =\frac{\widetilde{B}_k(x)}{k!},\]
integration by parts gives
\begin{align}
    \label{t2-lem1-Ik}
    I_k = -\frac{B_{k+1}}{(k+1)!(N+\alpha)^{s+k}} +(s+k)I_{k+1}.
\end{align}
Iterating \eqref{t2-lem1-Ik}, and using
\[B_{2j+1}=0 \qquad (j\ge 1),\]
we obtain
\[-sI_1 = \sum_{j=1}^M \frac{B_{2j}}{(2j)!}(s)_{2j-1} \frac{1}{(N+\alpha)^{s+2j-1}}+E_M,\]
where
\[E_M := E_M(s,\alpha,N)= - \frac{(s)_{2M+1}}{(2M+1)!}\int_N^\infty\frac{\widetilde{B}_{2M+1}(x)}{(x+\alpha)^{s+2M+1}}\mmd x.\]
Combining this with \eqref{t2-lem1-Euler} gives
\begin{align}
    \label{t2-lem1-Euler2}
    \zeta(s,\alpha)&=\sum_{n=0}^{N-1}\frac{1}{(n+\alpha)^s}+\frac{(N+\alpha)^{1-s}}{s-1}
    +\frac{1}{2(N+\alpha)^s} \nonumber \\
    &\quad+\sum_{j=1}^M \frac{B_{2j}(s)_{2j-1}}{(2j)!(N+\alpha)^{s+2j-1}}+E_M.
\end{align}

We now estimate the last three terms in \eqref{t2-lem1-Euler2}. Since \(N+\alpha\asymp T^A\) uniformly for \(0<\alpha\le 1\), we first have
\begin{align}
    \label{t2-lem1-error1}
    \Bg|\frac{1}{2(N+\alpha)^s}\Bg| \ll T^{-A\sigma} \le T^{-A\sigma_0}.
\end{align}
For \(1\leq j\leq M\), put \(u=2j-1\). Since \(|(s)_u| \ll_u T^u+(1+\sigma)^u\), we obtain
\[|(s)_u|(N+\alpha)^{-\sigma-u} \ll_u T^u(N+\alpha)^{-\sigma-u} +(1+\sigma)^u(N+\alpha)^{-\sigma-u}.\]
The first term on the right-hand side above satisfies
\[T^u(N+\alpha)^{-\sigma-u} \ll T^{-A\sigma_0-u(A-1)}.\]
On the other hand, for the second term, since \(N+\alpha\ge 2\) for sufficiently large \(T\),
\[(1+\sigma)^u(N+\alpha)^{-(\sigma-\sigma_0)} \ll_{u,\sigma_0} 1.\]
Hence,
\[(1+\sigma)^u(N+\alpha)^{-\sigma-u} \ll_{u,\sigma_0} T^{-A\sigma_0-Au}.\]
Combining the above estimates, we get
\begin{align}
    \label{t2-lem1-error2}
    \sum_{j=1}^M \frac{B_{2j}(s)_{2j-1}}{(2j)!(N+\alpha)^{s+2j-1}} \ll_{\sigma_0,A} T^{-A\sigma_0-(A-1)}.
\end{align}
It remains to estimate \(E_M\). The Fourier series of the periodic Bernoulli function gives
\[\bg|\widetilde{B}_{2M+1}(x)\bg| \le \frac{2(2M+1)!}{(2\pi)^{2M+1}}\zeta(2M+1).\]
Therefore,
\[|E_M|\ll_M |(s)_{2M+1}|(N+\alpha)^{-\sigma-2M}.\]
Using the same argument as above, we obtain
\[|E_M| \ll_{\sigma_0,A} T^{2M+1-A(\sigma_0+2M)} + T^{-A\sigma_0-2MA}.\]
Since
\[2M+1-A(\sigma_0+2M) = -(A\sigma_0+2M(A-1)-1)<0,\]
there exists a constant \(\delta=\delta(\sigma_0,A)>0\), such that
\begin{align}
    \label{t2-lem1-error3}
    |E_M| \ll_{\sigma_0,A} T^{-\delta}.
\end{align}
After decreasing \(\delta\) if necessary, we may assume that \(0<\delta<A\sigma_0\). Substituting \eqref{t2-lem1-error1}, \eqref{t2-lem1-error2} and \eqref{t2-lem1-error3} into \eqref{t2-lem1-Euler2} proves
\[\zeta(s,\alpha)=\sum_{n=0}^{N-1} \frac{1}{(n+\alpha)^s}+\frac{(N+\alpha)^{1-s}}{s-1}+O_{\sigma_0,A}(T^{-\delta})\]
and the proof is complete.
\end{proof}

\subsection{Proof of Theorem \ref{thm2}}
\label{t2-sec-prove}

Let
\[\lambda(\sigma) := \int_ 0^1 \frac{\mmd t}{2t^{-\sigma}-1}.\]
Fix a positive real number \(\kappa\) such that 
\[0<\kappa <\frac{\sigma-1/2}{\sigma(1+\lambda(\sigma))}.\]
Choosing \(1<A<(2(1-\sigma))\inv\) and sufficiently close to \(1\) so that 
\[\kappa\sigma(1+\lambda(\sigma))<\frac{1}{2}-A(1-\sigma).\]
Then, for large \(T\), put \(X=\kappa\log T \lo T\). Define \(r(n)\) to be a completely multiplicative function whose values at primes are given by
\begin{equation*}
    r(p) =
    \begin{cases}
        1- (p/X)^\sigma, ~ &\operatorname{if} p \le X, \\
		0, ~ &\operatorname{if} p > X.
    \end{cases}
\end{equation*}
Furthermore, define 
\[
R(t) := \prod_{p \le X}\bg(1-r(p)p^{\mit}\bg)\inv = \sum_{n\ge 1} r(n)n^{\mit}.
\]
The prime number theorem yields that
\begin{align}
    \label{t2-Rupper}
    |R(t)|^2 \le T^{2\kappa\sigma+o(1)}.
\end{align}

With the choice of \(A\), put \(N=\lfloor T^A \rfloor\). Applying Lemma \ref{t2-lem1} with \(\sigma_0=\sigma\) and \(\alpha=a/q\), we obtain, uniformly for \(\sqrt{T} \le t \le T\), that
\begin{align}
    \label{t2-Euler3}
    \zeta\Bg(\sigma+\mit,\frac{a}{q}\Bg) = \sum_{n=0}^{N-1}\frac{1}{(n+a/q)^{\sigma+\mit}}+\frac{(N+a/q)^{1-\sigma-\mit}}{\sigma-1+\mit}+O_{\sigma,A}\bg(T^{-\delta}\bg)
\end{align}
for some \(\delta>0\). Since
\[\Bg(n+\frac{a}{q}\Bg)^{-\sigma-\mit} = q^{\sigma+\mit}(q n+a)^{-\sigma-\mit},\]
and the integers \(qn+a\), \(0\le n<N\), are precisely the positive integers \(k\le qN\) satisfying \(k\equiv a\pmod q\), it follows that
\begin{align}
    \label{t2-equation}
    \sum_{n=0}^{N-1}\frac{1}{(n+a/q)^{\sigma+\mit}} = q^{\sigma+\mit} \sum_{\substack{k \le qN \\ k \equiv a \pmod q}} \frac{1}{k^{\sigma+\mit}}.
\end{align}
Moreover, since \(|t| \ge \sqrt{T}\), we get
\[\Bg|\frac{(N+a/q)^{1-\sigma-\mit}}{\sigma-1+\mit}\Bg|\ll_{\sigma,A}T^{A(1-\sigma)-1/2}.\]
Since \(A(1-\sigma)<1/2\), after decreasing \(\delta\) if necessary, \eqref{t2-Euler3} and \eqref{t2-equation} yield, uniformly for \(\sqrt{T} \le t \le T\), 
\begin{align}
    \label{t2-Euler4}
    \zeta\Bg(\sigma+\mit,\frac{a}{q}\Bg) = D(t)+ O_{\sigma,A}\bg(T^{-\delta}\bg),
\end{align}
where
\[D(t):= D(t,q,N) =  q^{\sigma+\mit} \sum_{\substack{k \le qN \\ k \equiv a \pmod q}} \frac{1}{k^{\sigma+\mit}}.\]
Trivially, we have
\begin{align}
    \label{t2-Dupper}
    |D(t)| \le q^\sigma \sum_{k \le qN} \frac{1}{k^\sigma} \ll_\sigma q^\sigma(qN)^{1-\sigma} \ll_{q,\sigma} T^{A(1-\sigma)+o(1)}.
\end{align}

We now apply the long resonator method to \(D(t)\). To this end, define the following four integrals:
\begin{align*}
    &U_1 :=U_1(R,T) = \int_{\sqrt{T}}^T |R(t)|^2 \Phi\Bg(\frac{\log T}{T}t\Bg) \mdt, \\
    &V_1 :=V_1(R,T) = \int_\mbr |R(t)|^2 \Phi\Bg(\frac{\log T}{T}t\Bg) \mdt, \\
    &U_2 :=U_2(R,T) =\int_{\sqrt{T}}^T D(t) |R(t)|^2 \Phi\Bg(\frac{\log T}{T}t\Bg) \mdt, \\
    &V_2 :=V_2(R,T) =\int_\mbr D(t)|R(t)|^2 \Phi\Bg(\frac{\log T}{T}t\Bg) \mdt.
\end{align*}
Throughout this section, let \(\Phi\) be the Gaussian function introduced in Section \ref{sec-pthm1}. The rapid decay of \(\Phi\) ensures the absolute convergence of all sums and integrals occurring below. Since \(D(-t)=\overline{D(t)}\), \(R(-t)=\overline{R(t)}\) and \(\Phi\) is even, the contributions from the positive and negative ranges are conjugate to each other. Furthermore, it is clear that we have
\begin{align}
    \label{t2-max1}
    \max_{\sqrt{T}\le t\le T} |D(t)| \ge \frac{|U_2|}{U_1} \ge \frac{\re U_2}{U_1}.
\end{align}

Combining \eqref{t2-Rupper} with \eqref{t2-Dupper} gives that
\[\Big| \int_{|t|\le \sqrt{T}} D(t) |R(t)|^2\Phi\Big(\frac{\log T}{T}t \Big)\mathrm{d}t\Big| \ll T^{2\kappa\sigma+A(1-\sigma)+1/2+ o(1)}.\]
Moreover, it follows from the rapid decay of \(\Phi(t)\) that
\[\Big| \int_{|t|\ge T } D(t) |R(t)|^2 \Phi\Big(\frac{\log T}{T}t \Big)\mathrm{d}t\Big| \ll 1.\]
Combining the two upper bounds above shows that
\begin{align}
    \label{t2-U2V2eq}
    2\re U_2 = V_2 + O\bg(T^{2\kappa\sigma+A(1-\sigma)+1/2+ o(1)} \bg).
\end{align}
Similarly, we have
\begin{align}
    \label{t2-U1V1eq}
    2U_1=V_1+ O\bg(T^{2\kappa\sigma+ 1/2+ o(1)} \bg).
\end{align}
The argument in \cite[pp. 78-79]{Dong2022phd} gives the following lower bound for \(V_1\):
\begin{equation}
    \label{t2-V1lower}
    V_1 = \int_\mbr \sum_{m,n \ge 1} r(m)r(n) \Bg(\frac{m}{n} \Bg)^{\mit}\Phi\Bg(\frac{\log T}{T}t \Bg)\mmd t  \gg T^{\kappa\sigma(1-\lambda(\sigma))+1+o(1)}.
\end{equation}
Moreover, by \eqref{t2-Dupper}, 
\[|V_2| \le T^{A(1-\sigma)+o(1)}V_1.\]
Hence, using \eqref{t2-max1}, \eqref{t2-U2V2eq}, \eqref{t2-U1V1eq} and \eqref{t2-V1lower}, we obtain
\begin{align}
    \label{t2-max2}
    \max_{\sqrt{T}\le t\le T} |D(t)| \ge\frac{V_2}{V_1} +O\bg(T^{\kappa\sigma(1+\lambda(\sigma))+A(1-\sigma) -1/2+ o(1)} \bg).
\end{align}
Recall the construction of \(\kappa\), we have
\begin{align}
    \label{t2-kappa}
        2\kappa\sigma+A(1-\sigma)+\frac{1}{2} < \kappa\sigma(1-\lambda(\sigma))+1.
\end{align}
Thus, by \eqref{t2-kappa}, the error term on the right-hand side of \eqref{t2-max2} is \(o(1)\). It therefore remains to obtain an effective lower bound for the ratio \(V_2/V_1\).

Expanding \(D(t)\) and \(|R(t)|^2\) in the definition of \(V_2\), we obtain
\[V_2 = \int_\mbr q^{\sigma+\mit} \sum_{\substack{k \le qN \\ k \equiv a \pmod q}}\frac{1}{k^{\sigma+\mit}}\sum_{m,n\ge 1}r(m)r(n)\Big(\frac{m}{n} \Big)^{\mit}\Phi\Big(\frac{\log T}{T}t \Big)\mmd t.\]
Since \(D(t)\) is a finite Dirichlet polynomial, \(R(t)\) is absolutely convergent, and \(\Phi\) decays rapidly, we may interchange the order of summation and integration. Hence
\[V_2 = \sum_{\substack{k \le qN \\ k \equiv a \pmod q}}\frac{q^\sigma}{k^\sigma} \sum_{m,n\ge 1}r(m)r(n) \int_\mbr \Big(\frac{qm}{kn} \Big)^{\mit}\Phi\Big(\frac{\log T}{T}t \Big)\mmd t.\]
Since \(r(n)\ge0\) and \(\whp(\xi) \ge 0\) for all \(\xi\in\mbr\), we may therefore retain only the terms for which \(m=kd\) and \(n=q\ell\) with \(d,\ell\ge1\). It follows that
\[V_2 \ge q^\sigma\sum_{\substack{k \le qN \\ k \equiv a \pmod q}}\frac{1}{k^\sigma}\sum_{d,\ell \ge 1}r(kd)r(q\ell)\int_\mbr \Big(\frac{qkd}{kq\ell} \Big)^{\mit}\Phi\Big(\frac{\log T}{T}t \Big)\mmd t,\]
Since \(r(n)\) is completely multiplicative, we have
\begin{align*}
    V_2 &\ge r(q)q^\sigma\sum_{\substack{k \le qN \\ k \equiv a \pmod q}}\frac{r(k)}{k^\sigma}\int_\mbr\sum_{d,\ell \ge 1} r(d)r(\ell) \Bg(\frac{d}{\ell}\Bg)^{\mit}\Phi\Big(\frac{\log T}{T}t \Big)\mathrm{d}t \\
    &=r(q)q^\sigma\sum_{\substack{k \le qN \\ k \equiv a \pmod q}}\frac{r(k)}{k^\sigma} \cdot V_1.
\end{align*}
Thus, 
\begin{align}
    \label{t2-V2V1ratiolower}
    \frac{V_2}{V_1} \ge r(q)q^\sigma \sum_{\substack{k \le qN \\ k \equiv a \pmod q}}\frac{r(k)}{k^\sigma}.
\end{align}

We next remove the restriction \(k\leq qN\) in \eqref{t2-V2V1ratiolower}. Choose \(\eta \in (0,\sigma-1/2)\). By Rankin's trick,
\[\sum_{k > qN}\frac{r(k)}{k^\sigma} \le (qN)^{-\eta}\sum_{k > qN}\frac{r(k)}{k^{\sigma-\eta}} \le (qN)^{-\eta}\sum_{k \ge 1}\frac{r(k)}{k^{\sigma-\eta}}.\]
Since \(r(n)\) is completely multiplicative and is supported on integers whose prime factors are at most \(X\), we have
\[\sum_{k \ge 1}\frac{r(k)}{k^{\sigma-\eta}} = \prod_{p \le X}\Bg(1-\frac{r(p)}{p^{\sigma-\eta}}\Bg)\inv.\]
Since \(\sigma-\eta>1/2\), it follows that
\begin{align*}
    \log \sum_{k \ge 1}\frac{r(k)}{k^{\sigma-\eta}} & =\sum_{p\le X}\frac{r(p)}{p^{\sigma-\eta}} +O_{\sigma,\eta}\Bg(\sum_{p \le X}\frac{1}{p^{2(\sigma-\eta)}}\Bg) \\
    &=\sum_{p \le X}\frac{1}{p^{\sigma-\eta}}-\frac{1}{X^\sigma}\sum_{p\le X}p^\eta +O_{\sigma,\eta}(1). 
\end{align*}
The prime number theorem yields that
\[\sum_{p \le X}\frac{1}{p^{\sigma-\eta}} = \Bg(\frac{1}{1-\sigma+\eta}+o(1)\Bg)\frac{X^{1-\sigma+\eta}}{\log X}\]
and
\[\frac{1}{X^\sigma}\sum_{p\le X}p^\eta = \Bg(\frac{1}{1+\eta}+o(1)\Bg)\frac{X^{1-\sigma+\eta}}{\log X}.\]
Consequently,
\[\log \sum_{k \ge 1}\frac{r(k)}{k^{\sigma-\eta}} =\Bg(\frac{1}{1-\sigma+\eta}-\frac{1}{1+\eta}+o(1)\Bg)\frac{X^{1-\sigma+\eta}}{\log X}.\]
Using the definition of \(X\) and \(1-\sigma+\eta<1/2\), we obtain
\[\sum_{k \ge 1}\frac{r(k)}{k^{\sigma-\eta}}=T^{o(1)}.\]
Since \(N=\lfloor T^A \rfloor\) and \(q\) is fixed, it follows that
\[\sum_{k>qN}\frac{r(k)}{k^\sigma}\ll T^{-A\eta+o(1)}.\]
Therefore,
\begin{align*}
    0 \le &\Bg( \sum_{\substack{k \ge 1 \\ k \equiv a \pmod q}} - \sum_{\substack{k \le qN \\ k \equiv a \pmod q}} \Bg)\frac{r(k)}{k^\sigma} \\
    & = \sum_{\substack{k > qN \\ k \equiv a \pmod q}}\frac{r(k)}{k^\sigma} \le \sum_{k >qN}\frac{r(k)}{k^\sigma}\ll T^{-A\eta+o(1)}.
\end{align*}
Since \(a\) is fixed and, for sufficiently large \(T\), \(r(a) \asymp_{a,\sigma} 1\), we have
\[\sum_{\substack{k \ge 1 \\ k \equiv a \pmod q}}\frac{r(k)}{k^\sigma} \gg_{a,\sigma} 1.\]
Hence,
\begin{align}
\label{t2-sumpmodeq}
    \sum_{\substack{k \le qN \\ k \equiv a \pmod q}}\frac{r(k)}{k^\sigma} = \bg(1+o(1)\bg)\sum_{\substack{k \ge 1 \\ k \equiv a \pmod q}}\frac{r(k)}{k^\sigma}.
\end{align}
Combining \eqref{t2-V2V1ratiolower} with \eqref{t2-sumpmodeq}, we obtain
\begin{align}
    \label{t2-V2V1ratiolower2}
    \frac{V_2}{V_1} \ge \bg(1+o(1)\bg)r(q)q^\sigma\sum_{\substack{k \ge 1 \\ k \equiv a \pmod q}}\frac{r(k)}{k^\sigma}.
\end{align}

We now estimate the sum on the right-hand side of \eqref{t2-V2V1ratiolower2}. By the orthogonality of Dirichlet characters,
\[\sum_{\substack{k \ge 1 \\ k \equiv a \pmod q}}\frac{r(k)}{k^\sigma} =\frac{1}{\phi(q)}\sum_{\chi \pmod q}\overline{\chi(a)}\sum_{k \ge1}\frac{r(k)\chi(k)}{k^\sigma}.\]
For convenience, define
\[Q_\sigma(\chi,X) := \sum_{k \ge 1}\frac{r(k)\chi(k)}{k^\sigma} = \prod_{p \le X}\Bg(1-\frac{\chi(p)r(p)}{p^\sigma}\Bg)\inv.\]
Since \((a,q)=1\), we have \(\chi_0(a)=1\), and hence
\begin{align}
    \label{t2-sumdes}
    \sum_{\substack{k \ge 1 \\ k \equiv a \pmod q}}\frac{r(k)}{k^\sigma} =\frac{1}{\phi(q)}\Bg( Q_\sigma(\chi_0,X)+\sum_{\substack{\chi \pmod q \\ \chi \neq \chi_0}}\overline{\chi(a)} Q_\sigma(\chi,X) \Bg).
\end{align}
We first consider the contribution from the principal character \(\chi_0\). By the definition of \(r(n)\), we have
\[Q_\sigma(\chi_0,X) = \prod_{\substack{p \le X \\ p \nmid q}}\Bg(1-\frac{r(p)}{p^\sigma}\Bg)\inv =\prod_{\substack{p \le X \\ p \nmid q}}\Bg(1-\frac{1}{p^\sigma}+\frac{1}{X^\sigma}\Bg)\inv.\]
Since \(\sigma \in(1/2,1)\), we have \(\sum_p p^{-2\sigma}<\infty\). Therefore, 
\[\log Q_\sigma(\chi_0,X) = \sum_{\substack{p \le X \\ p \nmid q}}\Bg(\frac{1}{p^\sigma}-\frac{1}{X^\sigma}\Bg)+O_\sigma(1) = \sum_{p \le X}\Bg(\frac{1}{p^\sigma}-\frac{1}{X^\sigma}\Bg)+O_{\sigma,q}(1),\]
where in the last step we used the fact that \(q\) is fixed. Using the prime number theorem and partial summation,
\[\sum_{p\le X}\frac{1}{p^\sigma} = \Bg(\frac{1}{1-\sigma}+o_\sigma(1)\Bg)\frac{X^{1-\sigma}}{\log X},\]
while
\[\sum_{p \le X}\frac{1}{X^\sigma}=\bg(1+o(1)\bg)\frac{X^{1-\sigma}}{\log X}.\]
It follows that
\[\log Q_\sigma(\chi_0,X) =\Bg(\frac{\sigma}{1-\sigma}+o_{\sigma,q}(1)\Bg)\frac{X^{1-\sigma}}{\log X}.\]
Hence,
\begin{align}
    \label{t2-Qchi0}
    Q_\sigma(\chi_0,X) = \exp \Bg(\Bg(\frac{\sigma}{1-\sigma}+o_{\sigma,q}(1)\Bg)\frac{X^{1-\sigma}}{\log X}\Bg).
\end{align}
We next consider the contribution from the non-principal characters. Let
\(\chi \neq \chi_0 \pmod q\). By the definition of \(Q_\sigma(\chi,X)\), we have
\begin{align}
\label{t2-logQchi}
    \log |Q_\sigma(\chi,X)| &= \re \sum_{p \le X} \log \Bg(1-\frac{\chi(p)r(p)}{p^\sigma}\Bg)\inv = \re \sum_{p \le X}\frac{\chi(p)r(p)}{p^\sigma}+O_\sigma(1) \nonumber \\
    &= \re\Bg(\sum_{p \le X}\frac{\chi(p)}{p^\sigma} - \frac{1}{X^\sigma}\sum_{p \le X} \chi(p)\Bg)+O_\sigma(1).
\end{align}
Let \(W(\chi,X) :=\sum_{p \le X}\chi(p)\). Since \(q\) is fixed, the Siegel–Walfisz theorem gives
\[W(\chi,X) = o_q\Bg(\frac{X}{\log X}\Bg)\]
for every non-principal character \(\chi \pmod q\). By partial summation, 
\[\sum_{p \le X}\frac{\chi(p)}{p^\sigma} = \frac{W(\chi,X)}{X^\sigma} +\sigma \int_2^X \frac{W(\chi,u)}{u^{\sigma+1}}\mmd u = o_{\sigma,q}\Bg(\frac{X^{1-\sigma}}{\log X}\Bg).\]
Moreover, 
\[\frac{1}{X^\sigma}\sum_{p \le X}\chi(p) =o_{\sigma,q}\Bg(\frac{X^{1-\sigma}}{\log X}\Bg).\]
Substituting these estimates into \eqref{t2-logQchi}, we obtain
\[|Q_\sigma(\chi,X)| = \exp\Bg(o_{\sigma,q}\Bg(\frac{X^{1-\sigma}}{\log X}\Bg) \Bg).
\]
Since there are only finitely many characters modulo the fixed modulus \(q\), it follows that
\begin{align}
    \label{t2-Qchiupper}
    \Bgg|\sum_{\substack{\chi \pmod q \\ \chi \neq \chi_0}}\overline{\chi(a)} Q_\sigma(\chi,X)\Bgg| \le \exp\Bg(o_{\sigma,q}\Bg(\frac{X^{1-\sigma}}{\log X}\Bg) \Bg).
\end{align}
Comparing \eqref{t2-Qchiupper} with \eqref{t2-Qchi0}, we see that the contribution from the non-principal characters is negligible compared with that of the principal character. Therefore, by \eqref{t2-sumdes}, 
\begin{align}
\label{t2-final}
    \sum_{\substack{k \ge 1 \\ k \equiv a \pmod q}}\frac{r(k)}{k^\sigma} =\exp \Bg(\Bg(\frac{\sigma}{1-\sigma}+o_{\sigma,q}(1)\Bg)\frac{X^{1-\sigma}}{\log X}\Bg).
\end{align}

Then we deal with \(r(q)\). Since \(q\) is fixed and \(X\to\infty\), every prime divisor of \(q\) is at most \(X\) for sufficiently large \(T\). Write
\[q = \prod_{p^\nu \parallel q} p^\nu.\]
Then, by the definition of \(r(n)\),
\[r(q) =\prod_{p^\nu \parallel q} \Bg(1-\Bg(\frac{p}{X}\Bg)^\sigma\Bg)^\nu = 1+O_{q,\sigma}\Bg(\frac{1}{X^\sigma}\Bg).\]
In particular,
\begin{align}
    \label{t2-rqqsigma}
    r(q)q^\sigma = \exp\bg(O_{\sigma,q}(1)\bg).
\end{align}
Combining \eqref{t2-V2V1ratiolower2}, \eqref{t2-final} and \eqref{t2-rqqsigma}, we obtain
\begin{align}
    \label{t2-V2V1ratiolower3}
    \frac{V_2}{V_1} \ge \exp \Bg(\Bg(\frac{\sigma}{1-\sigma}+o_{\sigma,q}(1)\Bg)\frac{X^{1-\sigma}}{\log X}\Bg).
\end{align}
Substituting \eqref{t2-V2V1ratiolower3} into \eqref{t2-max2}, and using \eqref{t2-kappa}, we have
\[\max_{\sqrt{T}\le t\le T} |D(t)| \ge \exp \Bg(\Bg(\frac{\sigma}{1-\sigma}+o_{\sigma,q}(1)\Bg)\frac{X^{1-\sigma}}{\log X}\Bg).\]
Recalling that \(X=\kappa\log T\lo T \), we have
\[\frac{X^{1-\sigma}}{\log X} = \bg(\kappa^{1-\sigma}+o(1)\bg)\frac{(\log T)^{1-\sigma}}{(\lo T)^\sigma}.\]
It follows that 
\begin{align}
    \label{t2-max3}
    \max_{\sqrt{T}\le t\le T} |D(t)| \ge \exp \Bg(\Bg(\frac{\sigma}{1-\sigma}\kappa^{1-\sigma}+o_{\sigma,q}(1)\Bg) \frac{(\log T)^{1-\sigma}}{(\lo T)^\sigma}\Bg)
\end{align}
for sufficiently large \(T\).

Finally, recalling the approximation obtained at the beginning of the proof, that is, \eqref{t2-Euler4}, we conclude that
\[\max_{\sqrt{T} \le t \le T}\Bg|\zeta\Bg(\sigma+\mit,\frac{a}{q}\Bg)\Bg| \ge \exp\Bg( \Bg(\frac{\sigma}{1-\sigma}\kappa^{1-\sigma}+o_{\sigma,q}(1)\Bg) \frac{(\log T)^{1-\sigma}}{(\lo T)^\sigma}\Bg).\]
This completes the proof of Theorem \ref{thm2}.

\section{Proof of Theorem \ref{thm3}}
\label{sec-pthm3}

In this section, we prove Theorem \ref{thm3} by using the long resonance method in \cite{Aistleitner2019IMRN}. One of the main tools is a truncated approximation for the Hurwitz zeta function on the \(1\)-line, which follows from \eqref{fundamental-decomposition} together with a truncated Euler product for Dirichlet \(L\)-functions.

\subsection{An auxiliary lemma}
\label{t3-sec-lemma}

For \(y \ge 2\), define
\[L(s,\chi;y) = \prod_{p \le y}\Bg(1-\frac{\chi(p)}{p^s} \Bg)\inv =\sum_{\substack{k \ge 1 \\ P^+(k) \le y}} \frac{\chi(k)}{k^s}.\]
Then we have the following result.

\begin{lem}
    \label{t3-lem1}
Let \(q\) be fixed and let \((a,q)=1\). Set \(Y = \exp\bg((\log T)^{10}\bg)\). Then, uniformly for \(\sqrt{T}\le |t| \le T\), we have
\[\zeta\Bg(1+\mit,\frac{a}{q}\Bg) = \frac{q^{1+\mit}}{\phi(q)} \sum_{\chi\pmod q} \overline{\chi(a)}L(1+\mit,\chi;Y)+O_q\Bg(\frac{1}{(\log T)^{9}}\Bg).\]
\end{lem}

\begin{proof}
By the truncated Euler product estimate for Dirichlet \(L\)-functions, we have, uniformly for all characters \(\chi\pmod q\),
\[L(1+\mit,\chi) = L(1+\mit,\chi;Y) \Bg(1+O_q\Bg(\frac{1}{(\log T)^{10}}\Bg)\Bg)\]
whenever \(T^{1/10} \le |t| \le T\); see Dixit and Mahatab \cite[Lemma 3.1]{Dixit2021BAustMS}. In particular, this holds in the range \(\sqrt{T}\le |t| \le T\). Moreover, for fixed \(q\), the following classical upper bound holds uniformly in this range:
\[L(1+\mit,\chi) \ll_q \log T.\]
It also follows that
\[L(1+\mit,\chi;Y) \ll_q \log T\]
and
\[L(1+\mit,\chi)=L(1+\mit,\chi;Y)+O_q\Bg(\frac{1}{(\log T)^9}\Bg).\]
Substituting this estimate into \eqref{fundamental-decomposition}, and noting that the number of characters modulo the fixed modulus \(q\) is \(O_q(1)\), we obtain
\[\zeta\Bg(1+\mit,\frac{a}{q}\Bg) = \frac{q^{1+\mit}}{\phi(q)} \sum_{\chi\pmod q} \overline{\chi(a)}L(1+\mit,\chi;Y)+O_q\Bg(\frac{1}{(\log T)^{9}}\Bg).\]
This completes the proof.
\end{proof}

\subsection{Proof of Theorem \ref{thm3}}
\label{t3-sec-prove}

We now turn to the proof of Theorem \ref{thm3}. Following the long resonance method, for large \(T\), let \(X=(\log T \lo T)/6\). Define \(r(n)\) to be a completely multiplicative function whose values at primes are given by
\begin{equation*}
    r(p) =
    \begin{cases}
        1- p/X, ~ &\operatorname{if} p \le X, \\
		0, ~ &\operatorname{if} p > X.
    \end{cases}
\end{equation*}
Furthermore, define the resonator
\[R(t) := \prod_{p \le X}\bg(1-r(p)p^{\mit}\bg)\inv = \sum_{n\ge 1} r(n)n^{\mit}.\]
Since
\[|1-r(p)p^{\mit}| \ge 1-r(p) = \frac{p}{X},\]
the prime number theorem gives
\begin{align}
    \label{t3-Rupper}
    |R(t)|^2 \le \exp\Bg(2\sum_{p \le X} (\log X-\log p)\Bg) \le T^{1/3+o(1)}.
\end{align}

Throughout this section, let \(\Phi\) be the Gaussian function introduced in Section \ref{sec-pthm1}. Similarly, the rapid decay of \(\Phi\) ensures the absolute convergence of all sums and integrals occurring below. For \(y \ge 2\), define
\[M(s;y) := M(s,a,q;y)= \frac{q^s}{\phi(q)} \sum_{\chi \pmod q}\overline{\chi(a)} L(s,\chi;y) .\]
Put \(Y=\exp\bg((\log T)^{10}\bg)\). By Lemma \ref{t3-lem1}, uniformly for \(\sqrt{T} \le |t| \le T\), 
\begin{align}
    \label{t3-zetaM}
    \zeta\Bg(1+\mit,\frac{a}{q}\Bg) = M(1+\mit;Y)+ O_q\Bg(\frac{1}{(\log T)^{9}}\Bg).
\end{align}
To apply the long resonance method to \(M(1+\mit;Y)\), we introduce the following four integrals:
\begin{align*}
    &I_1 :=I_1(R,T) = \int_{\sqrt{T}}^T |R(t)|^2 \Phi\Bg(\frac{\log T}{T}t\Bg) \mdt, \\
    &J_1 :=J_1(R,T) = \int_\mbr |R(t)|^2 \Phi\Bg(\frac{\log T}{T}t\Bg) \mdt, \\
    &I_2(y) :=I_2(R,T;y) =\int_{\sqrt{T}}^T M(1+\mit;y) |R(t)|^2 \Phi\Bg(\frac{\log T}{T}t\Bg) \mdt, \\
    &J_2(y) :=J_2(R,T;y) =\int_\mbr M(1+\mit;y) |R(t)|^2 \Phi\Bg(\frac{\log T}{T}t\Bg) \mdt.
\end{align*}
Since \(I_1>0\), we have
\begin{align}
    \label{t3-max1}
    \max_{\sqrt{T}\le t\le T} |M(1+\mit;Y)| \ge \frac{|I_2(Y)|}{I_1} \ge \frac{\re I_2(Y)}{I_1}.
\end{align}

We now extend the range of integration to the entire real line. By the definition of \(L(s,\chi;Y)\) and Mertens' theorem, uniformly for all \(t \in \mbr\),
\[|L(1+\mit,\chi;Y)|\le \prod_{p \le Y}\Bg(1-\frac{1}{p}\Bg)\inv\ll\log Y \ll (\log T)^{10}.\]
Thus,
\begin{align}
    \label{t3-Mupper}
    M(1+\mit;Y)\ll_q (\log T)^{10}.
\end{align}
Combining \eqref{t3-Mupper} with \eqref{t3-Rupper} implies
\[\Big| \int_{|t|\le \sqrt{T}} M(1+\mit;Y) |R(t)|^2\Phi\Big(\frac{\log T}{T}t \Big)\mmd t\Big| \ll T^{5/6+o(1)}.\]
Moreover, by the rapid decay of \(\Phi(t)\),
\[\Big| \int_{|t|\ge T } M(1+\mit;Y) |R(t)|^2 \Phi\Big(\frac{\log T}{T}t \Big)\mmd t\Big| \ll 1.\]
Since \(M(1-\mit;Y)=\overline{M(1+\mit;Y)}\) and \(\Phi\) is even, while \(|R(-t)|^2=|R(t)|^2\), the contributions from \([-T,-\sqrt T]\) and \([\sqrt T,T]\) are conjugate to each other. Hence, the above estimates give
\begin{align}
    \label{t3-I2J2eq}
    2\re I_2(Y) = J_2(Y) + O\big(T^{5/6+o(1)}\big).
\end{align}
Similarly, we have
\[\int_{|t|\le \sqrt{T}} |R(t)|^2\Phi\Big(\frac{\log T}{T}t \Big)\mmd t \ll T^{5/6+o(1)},\]
while the contribution from \(|t|\ge T\) is \(O(1)\). Therefore, 
\begin{align}
    \label{t3-I1J1eq}
    2I_1=J_1+O\big(T^{5/6+o(1)}\big).
\end{align}
By the same argument as in \cite[Eq. (8)]{Aistleitner2019IMRN}, we have the following lower bound for \(J_1\):
\begin{align}
    \label{t3-I1lower}
    J_1 = \int_\mbr \sum_{m,n\ge 1}r(m)r(n) \Bg(\frac{m}{n}\Bg)^{\mit}\Phi\Bg(\frac{\log T}{T}t\Bg) \mdt  \gg T^{1+o(1)}.
\end{align}
Moreover, \eqref{t3-Mupper} gives
\[|J_2(Y)|\ll_q (\log T)^{10}J_1.\]
Hence, by \eqref{t3-max1}, \eqref{t3-I2J2eq}, \eqref{t3-I1J1eq} and \eqref{t3-I1lower}, we obtain
\begin{align}
    \label{t3-max2}
    \max_{\sqrt{T}\le t \le T} |M(1+\mit;Y)| \ge \frac{J_2(Y)}{J_1}+ O\big(T^{-1/6+o(1)}\big).
\end{align}

Since the completely multiplicative function \(r(n)\) appearing in the resonator \(R(t)\) is supported on the set \(\{k:P^+(k) \le X\}\), the range of primes in \(J_2(Y)\) can be restricted to \(p \le X\). To this end, using the orthogonality of Dirichlet characters, we rewrite \(M(s;y)\) as follows:
\[M(s;y)=\frac{q^s}{\phi(q)} \sum_{\chi \pmod q}\overline{\chi(a)} \sum_{\substack{k \ge 1 \\ P^+(k) \le y}} \frac{\chi(k)}{k^s} = q^s \sum_{\substack{k \ge 1,P^+(k) \le y\\ k \equiv a \pmod q}} \frac{1}{k^s}.\]
Expanding \(M(1+\mit;y)\) and \(|R(t)|^2\), and interchanging the order of summation and integration, we obtain
\begin{align}
\label{t3-J2equation}
    J_2(y) &=q\sum_{\substack{k \ge 1,P^+(k) \le y\\ k \equiv a \pmod q}} \frac{1}{k}\sum_{m,n\ge1}r(m)r(n) \int_\mbr \Bg(\frac{qm}{kn}\Bg)^{\mit}\Phi\Bg(\frac{\log T}{T}t\Bg)\mmd t \nonumber \\
    &= q\frac{T}{\log T} \sum_{\substack{k \ge 1,P^+(k) \le y\\ k \equiv a \pmod q}} \frac{1}{k} \sum_{m,n\ge1}r(m)r(n) \widehat{\Phi}\Bg( \frac{T}{\log T}\log \frac{qm}{kn}\Bg).
\end{align}
Note that \(r(n)\ge 0\) and \(\whp(\xi)>0\) for all \(\xi\in\mbr\), every term in \eqref{t3-J2equation} is non-negative. Since \(X \le Y\) and 
\[\{k:P^+(k) \le X\} \subset \{k:P^+(k) \le Y\},\]
it follows that \(J_2(Y) \ge J_2(X)\).
Thus, combining this with \eqref{t3-max2}, we obtain
\begin{align}
    \label{t3-max3}
        \max_{\sqrt{T}\le t\le T} |M(1+\mit;Y)| \ge \frac{J_2(X)}{J_1}+ O\big(T^{-1/6+o(1)}\big).
\end{align}

It remains to obtain an effective lower bound for \(J_2(X)/J_1\). From \eqref{t3-J2equation} with \(y=X\), and using again the non-negativity of all terms, we may retain only the terms satisfying \(m=kd \) and \(n=q\ell\) with \(d,\ell \ge 1\). Hence, we get

\begin{align*}
    J_2(X) &\ge q \sum_{\substack{k \ge 1, P^+(k) \le X\\ k \equiv a \pmod q}} \frac{1}{k}\sum_{d,\ell \ge1} r(kd)r(q\ell)\int_\mbr \Bg(\frac{qkd}{kq\ell}\Bg)^{\mit}\Phi\Bg(\frac{\log T}{T}t\Bg) \mdt  \\
    &=qr(q) \sum_{\substack{k \ge 1, P^+(k) \le X\\ k \equiv a \pmod q}} \frac{r(k)}{k} \sum_{d,\ell \ge 1}r(d)r(\ell) \int_\mbr\Bg(\frac{d}{\ell}\Bg)^{\mit}\Phi\Bg(\frac{\log T}{T}t\Bg) \mdt.
\end{align*}
Here, in the last step, we use the fact that \(r(n)\) is completely multiplicative. By the definition of \(J_1\), we have
\[J_2(X) \ge qr(q) \sum_{\substack{k \ge 1, P^+(k) \le X\\ k \equiv a \pmod q}} \frac{r(k)}{k}J_1. \]
Thus,
\[\frac{J_2(X)}{J_1}\ge q r(q) \sum_{\substack{k \ge 1, P^+(k) \le X\\ k \equiv a \pmod q}}\frac{r(k)}{k}.\]
By the orthogonality of Dirichlet characters,
\[\sum_{\substack{k \ge 1, P^+(k) \le X\\ k \equiv a \pmod q}}\frac{r(k)}{k} = \frac{1}{\phi(q)}\sum_{\chi \pmod q}\overline{\chi(a)}P(\chi,X),\]
where
\[P(\chi,X) := \sum_{\substack{k \ge 1 \\ P^+(k) \le X}} \frac{\chi(k)r(k)}{k} = \prod_{p \le X}\Bg(1-\frac{\chi(p)r(p)}{p}\Bg)\inv.\]
Thus,
\begin{align}
    \label{t3-J2J1lower}
    \frac{J_2(X)}{J_1}\ge \frac{qr(q)}{\phi(q)} \sum_{\chi \pmod q}\overline{\chi(a)}P(\chi,X).
\end{align}

Since \((a,q)=1\), we have \(\chi_0(a)=1\). Hence, the right-hand side of \eqref{t3-J2J1lower} can be written as
\begin{align}
    \label{t3-decomposition}
    \frac{qr(q)}{\phi(q)} \bigg( P(\chi_0,X)+\sum_{\substack{\chi \pmod q \\ \chi \neq \chi_0}}\overline{\chi}(a) P(\chi,X) \bigg).
\end{align}
We first consider the contribution from the principal character. By the definition of \(r(n)\),
\begin{align}
    \label{t3-Pchi0eq}
    P(\chi_0,X) &= \prod_{\substack{p \le X \\ p \nmid q}}\Bg(1-\frac{r(p)}{p}\Bg)\inv =\prod_{\substack{p \le X \\ p \nmid q}} \Bg(1-\frac{1}{p}+\frac{1}{X}\Bg)\inv   \nonumber \\ 
    &= \prod_{\substack{p \le X \\ p \nmid q}} \Bg(1-\frac{1}{p}\Bg)\inv  \prod_{\substack{p \le X \\ p \nmid q}} \Bg(1+\frac{1}{X(1-1/p)}\Bg)\inv.
\end{align}
Since \(q\) is fixed, all prime divisors of \(q\) are at most \(X\) for sufficiently large \(T\). Using Mertens' theorem shows that
\begin{align}
    \label{t3-p1eq}
    \prod_{\substack{p \le X \\ p \nmid q}} \Bg(1-\frac{1}{p}\Bg)\inv =\prod_{p \le X} \Bg(1-\frac{1}{p}\Bg)\inv  \prod_{p \mid q}  \Bg(1-\frac{1}{p}\Bg) 
    = \frac{\phi(q)}{q}\mme^\gamma \log X + O_q(1).
\end{align}
On the other hand,
\[\log \prod_{\substack{p \le X \\ p \nmid q}}  \Bg(1+\frac{1}{X(1-1/p)}\Bg)\inv \ll \sum_{p \le X}\frac{1}{X}.\]
The prime number theorem gives that
\begin{align}
    \label{t3-p2eq}
    \prod_{\substack{p \le X \\ p \nmid q}}  \Bg(1+\frac{1}{X(1-1/p)}\Bg)\inv =1+O\Bg(\frac{1}{\log X}\Bg).
\end{align}
Combining \eqref{t3-Pchi0eq}, \eqref{t3-p1eq} and \eqref{t3-p2eq}, it follows that
\begin{align}
    \label{t3-pchi0estimate}
    P(\chi_0,X)=\frac{\phi(q)}{q}\mme^\gamma \log X + O_q(1).
\end{align}
For the non-principal Dirichlet characters \(\chi \pmod q\), by the definition of \(r(n)\), we have
\begin{align}
    \label{t3-Pchieq}
    P(\chi,X) &= \prod_{p \le X}\Bg(1-\frac{\chi(p)}{p}+\frac{\chi(p)}{X}\Bg)\inv \nonumber \\
    & =  \prod_{p \le X}\Bg(1-\frac{\chi(p)}{p}\Bg)\inv  \prod_{p \le X}\Bg(1+\frac{\chi(p)}{X(1-\chi(p)/p)}\Bg)\inv.
\end{align}
It follows from the trivial bound 
\[\Bg|1-\frac{\chi(p)}{p}\Bg| \ge 1-\frac{1}{p} \ge \frac{1}{2}\]
that
\[\Bg|\frac{\chi(p)}{X(1-\chi(p)/p)}\Bg| \le \frac{2}{X}.\]
Therefore, using the prime number theorem, we get
\[\log\prod_{p \le X}\Bg(1+\frac{\chi(p)}{X(1-\chi(p)/p)}\Bg)\inv \ll \frac{1}{\log X}, \]
that is,
\begin{align}
    \label{t3-p4eq}
    \prod_{p \le X}\Bg(1+\frac{\chi(p)}{X(1-\chi(p)/p)}\Bg)\inv = 1+O\Bg(\frac{1}{\log X}\Bg).
\end{align}
On the other hand, since \(\chi\) is non-principal and \(q\) is fixed, the partial Euler product satisfies 
\[ \prod_{p \le X}\Bg(1-\frac{\chi(p)}{p}\Bg)\inv  = L(1,\chi) +o_q(1).\]
Combining this with \eqref{t3-Pchieq} and \eqref{t3-p4eq} yields 
\[P(\chi,X) = L(1,\chi) +o_q(1) =O_q(1).  \]
Since there are only finitely many non-principal characters modulo \(q\), and \(|\chi(a)|=1\), it follows that
\begin{align}
    \label{t3-Pchiestimate}
    \sum_{\substack{\chi \pmod q \\ \chi \neq \chi_0}}\overline{\chi(a)} P(\chi,X) = O_q(1).
\end{align}

Finally, we deal with \(r(q)\). Since \(q\) is fixed, we can write
\[q = \prod_{p^\nu \parallel q} p^\nu.\]
For sufficiently large \(T\), every prime divisor of \(q\) is at most \(X\), and hence
\begin{align}
\label{t3-rqeq}
    r(q) =\prod_{p^\nu \parallel q} \Bg(1-\frac{p}{X}\Bg)^\nu = 1+O_q\Bg(\frac{1}{X}\Bg).
\end{align}
Combining \eqref{t3-J2J1lower}, \eqref{t3-decomposition}, \eqref{t3-pchi0estimate}, \eqref{t3-Pchiestimate} and \eqref{t3-rqeq} gives that
\[\frac{J_2(X)}{J_1} \ge \mme^\gamma \log X -C_q\]
for some positive constant \(C_q\).

Recalling \eqref{t3-max3}, we obtain
\[\max_{\sqrt{T}\le t \le T}|M(1+\mit;Y)|\ge \mme^\gamma \log X -C_q.\]
Since \(X=(\log T \lo T)/6\), we have
\[\max_{\sqrt{T}\le t \le T} |M(1+\mit;Y)| \ge \mme^\gamma(\lo T+\lt T)-C_q.\]
Finally, employing \eqref{t3-zetaM} gives 
\[\max_{\sqrt{T}\le t\le T}\Bg|\zeta\Bg(1+\mit,\frac{a}{q}\Bg)\Bg| \ge \mme^\gamma (\lo T+\lt T) -C_{q}.\]
This completes the proof of Theorem \ref{thm3}.

\section*{Acknowledgments}
Qiyu Yang was supported by the Natural Science Foundation of Henan Province (Grant No. 252300421782) and the National Natural Science Foundation of China (Grant No. 12601011). Guang-Liang Zhou was supported by the National Natural Science Foundation of China (Grant No. 12401009).
	
	\bibliographystyle{siam}
    \bibliography{reference}

@article{Aistleitner2016MathAnn,
  author = {C. Aistleitner},
  title = {Lower bounds for the maximum of the {R}iemann zeta function along vertical lines},
  journal = {Math. Ann.},
  volume = {365},
  number = {1-2},
  pages = {473--496},
  year = {2016}
}

@article{Aistleitner2019IMRN,
  title={Extreme values of the {R}iemann zeta function on the 1-line},
  author={Aistleitner, Christoph and Mahatab, Kamalakshya and Munsch, Marc},
  journal={Int. Math. Res. Not. IMRN},
  volume={2019},
  number={22},
  pages={6924--6932},
  year={2019}
}

@article{Aistleitner2019QJMath,
  author = {C. Aistleitner and K. Mahatab and M. Munsch and A. Peyrot},
  title = {On large values of \({L}(\sigma, \chi)\)},
  journal = {Q. J. Math.},
  volume = {70},
  number = {3},
  pages = {831--848},
  year = {2019}
}

@book{Apostol1976,
    author = {T. M. Apostol},
    title = {Introduction to {A}nalytic {N}umber {T}heory},
    series = {Undergraduate Texts in Mathematics},
      publisher = {Springer-Verlag},
  address = {New York},
  year = {1976}
}

@article{Balasubrmanian1977PIAS,
    author = {R. Balasubramanian and K. Ramachandra},
    title = {On the frequency of {T}itchmarsh's phenomenon for $\zeta(s)$-$\mathrm{\uppercase\expandafter{\romannumeral 3}}$},
    journal = {Proc. Indian Acad. Sci.},
    year = {1977},
    volume = {86 A},
    pages = {341--351}
}

@article{Bondarenko2023BLMS,
  title={A dichotomy for extreme values of zeta and {D}irichlet {$L$}-functions},
  author={A. Bondarenko and P. Darbar and M. V. Hagen and W. Heap and K. Seip},
  journal={Bull. Lond. Math. Soc.},
  volume={55},
  number={6},
  pages={2963--2975},
  year={2023}
}

@article{Bondarenko2017DukeMJ,
  author = {A. Bondarenko and K. Seip},
  title = {Large greatest common divisor sums and extreme values of the {R}iemann zeta function},
  journal = {Duke Math. J.},
  volume = {166},
  number = {9},
  pages = {1685--1701},
  year = {2017}
}

@article{Bondarenko2018MathAnn,
    author = {Bondarenko, A and Seip, K},
    title = {Extreme values of the {R}iemann zeta function and its argument},
    journal = {Math. Ann.},
    volume = {372},
    number = {3-4},
    pages = {999--1015},
    year = {2018}
}

@article{Bondarenko2018operTAA,
  title={Note on the resonance method for the {R}iemann zeta function},
  author={Bondarenko, Andriy and Seip, Kristian},
  journal = {Oper. Theory Adv. Appl.},
  volume = {261},
  pages={121--139},
  year={2018}
}

@article{delaBreteche2019PLMS,
    author = {R. de la Bret{\`e}che and G. Tenenbaum},
    title = {Sommes de {G}\'al et applications},
    journal = {Proc. Lond. Math. Soc.},
    volume={119},
    pages={104--134},
    year = {2019}
}

@article{Dixit2021BAustMS,
    author = {Dixit, A. B. and  Mahatab, K.},
    title = {Large values of ${L}$-functions on the $1$-line},
    journal = {Bull. Aust. Math. Soc.},
    year = {2021},
    volume = {103},
    number ={2},
    pages = {230--243}
}

@misc{Dong2022phd,
  author = {Z. Dong},
  title = {Distribution of values of the {R}iemann zeta function},
  howpublished = {Université Paris-Est Créteil Val-de-Marne - Paris 12},
  year = {2022}
}

@article{Heap2025Crelle,
    author = {W. Heap and A. Sahay},
    title = {The fourth moment of the {H}urwitz zeta function},
    journal = {J. reine angew. Math.},
    year = {2025},
    volume ={818},
    pages ={291--319}
}

@article{Zhli2026JAustMS,
  title={Extreme values of derivatives of the {D}edekind zeta function of a cyclotomic field},
  author={Li, Zhonghua and Song, Yutong and Yang, Qiyu and Zhao, Shengbo},
  journal={J. Aust. Math. Soc.},
  pages={1--27},
  year={2026},
  publisher={Cambridge University Press}
}

@article{Montgomery1977CommentMH,
    author={Montgomery, Hugh L.},
    title ={Extreme values of the {R}iemann zeta function},
    journal ={Comment. Math. Helv.},
    volume={52},
    page={511--518},
    year = {1977}
}

@article{Rademacher1959MathZ,
    author = {H. Rademacher},
    title = {On the {P}hragm\'en-{L}indel\"of theorem and some applications},
    journal = {Math. Z.},
    year = {1959},
    volume = {72},
    pages ={192--204}
}

@article{Ramachandra1989ArchMath,
    author = {K. Ramachandra and A. Sankaranarayanan},
    title = {Omega-theorems for the {H}urwitz zeta-function},
    journal = {Arch. Math.},
    year = {1989},
    volume = {53},
    pages = {469--481}
}

@article{Ramachandra1991AA,
    author = {Kanakanahalli Ramachandra and Ayyadurai Sankaranarayanan},
    title = {Note on a paper by {H}. {L}. {M}ontgomery-$\mathrm{\uppercase\expandafter{\romannumeral 2}}$},
    volume={58},
  pages={299--308},
  year={1991},
  journal = {Acta Arith.}
}

@article{Rane1980JLMS,
    author = {V. V. Rane},
    title = {On the mean square value of {D}irichlet {$L$}-series},
    journal = {J. Lond. Math. Soc.},
    year = {1980},
    volume = {21},
    number = {2},
    pages = {203--215}
}

@article{Sahay2023MPCambridgePS,
    author = {A. Sahay},
    title = {Moments of the {H}urwitz zeta function on the critical line},
    journal = {Math. Proc. Cambridge Philos. Soc.},
    year = {2023},
    volume = {174},
    number ={3},
    pages ={631--661}
}

@article{Soundararajan2008MathAnn,
  author = {K. Soundararajan},
  title = {Extreme values of zeta and {$L$}-functions},
  journal = {Math. Ann.},
  volume = {342},
  number = {2},
  pages = {467--486},
  year = {2008}
}

@article{Voronin1988IANSSSR,
  author = {S. M. Voronin},
  title = {Lower bounds in {R}iemann zeta-function theory},
  journal = {Izv. Akad. Nauk SSSR Ser. Mat.},
  volume = {52},
  number = {4},
  pages = {882--892, 896},
  year = {1988}
}

@article{Ddyang2022Mathematika,
  author = {D. Yang},
  title = {Extreme values of derivatives of the {R}iemann zeta function},
  journal = {Mathematika},
  volume = {68},
  number = {2},
  pages = {486-510},
  year = {2022}
}

@article{Ddyang2023arxiv,
  title = {Omega theorems for logarithmic derivatives of zeta and ${L}$-functions},
  author = {Daodao Yang},
  journal = {Preprint, arXiv:2311.16371},
  year = {2023}
}

@article{Ddyang2024BLMS,
  title={Extreme values of derivatives of zeta and {$L$}-functions},
  author={D. Yang},
  journal={Bull. Lond. Math. Soc.},
  volume={56},
  number={1},
  pages={79--95},
  year={2024}
}

@article{Qyyang2024JNT,
    title = {Large values of $\zeta(s)$ for $1/2<${R}e$(s)<1$},
journal = {J. Number Theory},
volume = {254},
pages = {199-213},
year = {2024},
author ={Qiyu Yang}
}
\end{document}